\documentclass[11pt,reqno]{amsart}
\usepackage{amssymb,amsmath,amsthm,enumitem}
\usepackage[all,cmtip]{xy}

\usepackage[pdfstartview=FitH]{hyperref}

\usepackage{fullpage}
\allowdisplaybreaks

\numberwithin{equation}{section}

\theoremstyle{plain}
\newtheorem{thm}{Theorem}[section]
\newtheorem{lem}[thm]{Lemma}
\newtheorem{prop}[thm]{Proposition}
\newtheorem{cor}[thm]{Corollary}

\newtheorem*{main}{Main Theorem}

\theoremstyle{definition}
\newtheorem{defn}[thm]{Definition}

\newtheorem*{hyp}{Hypothesis}

\theoremstyle{remark}
\newtheorem*{rem}{Remark}

\newcommand{\mb}{\mathbf}
\newcommand{\mc}{\mathcal}
\newcommand{\mf}{\mathfrak}

\newcommand{\Z}{\mathbf{Z}}
\newcommand{\Q}{\mathbf{Q}}
\newcommand{\R}{\mathbf{R}}
\newcommand{\C}{\mathbf{C}}

\newcommand{\A}{\mathbf{A}}
\renewcommand{\O}{\mathcal{O}}

\newcommand{\GL}{\mathrm{GL}}
\newcommand{\BC}{\mathrm{BC}}
\newcommand{\cusp}{\mathrm{cusp}}
\newcommand{\U}{\mathrm{U}}
\renewcommand{\sp}{\mathrm{sp}}
\newcommand{\Iw}{\mathrm{Iw}}
\newcommand{\As}{\mathrm{As}}
\newcommand{\id}{\mathrm{id}}

\newcommand{\ord}{\mathrm{ord}}
\newcommand{\SL}{\mathrm{SL}}
\newcommand{\SU}{\mathrm{SU}}
\newcommand{\alg}{\mathrm{alg}}

\DeclareMathOperator{\Sym}{Sym}
\DeclareMathOperator{\ad}{ad}
\DeclareMathOperator{\tr}{tr}
\DeclareMathOperator{\Hom}{Hom}
\DeclareMathOperator{\End}{End}

\DeclareMathOperator{\Spec}{Spec}
\DeclareMathOperator{\Ev}{Ev}
\DeclareMathOperator{\Frac}{Frac}

\title{A $p$-adic $L$-function for non-critical adjoint $L$-values}
\author{Pak-Hin Lee}
\thanks{The author is supported by EPSRC Standard Grant EP/S020977/1 and Croucher Fellowship for Postdoctoral Research.}

\address{Mathematics Institute, University of Warwick, Coventry, CV4 7AL, United Kingdom}
\email{Pak-Hin.Lee@warwick.ac.uk}
\date{\today}

\begin{document}

\begin{abstract}
Let $K$ be an imaginary quadratic field, with associated quadratic character $\alpha$. We construct an analytic $p$-adic $L$-function interpolating the twisted adjoint $L$-values $L(1, \ad(f) \otimes \alpha)$ as $f$ varies in a Hida family; these special values are non-critical in the sense of Deligne. Our approach is based on Greenberg--Stevens' idea of $\Lambda$-adic modular symbols, which considers cohomology with values in a space of $p$-adic measures.
\end{abstract}

\maketitle

\tableofcontents

\section{Introduction}

\subsection{Background}

Let $f \in S_k(N, \chi)$ be a primitive cuspidal eigenform of weight $k \geq 2$, level $N$, and Nebentype $\chi$. The adjoint $L$-function of $f$ is defined as the Euler product
\[
L(s, \ad(f)) = \prod_\ell L_\ell(s, \ad(f)),
\]
where
\[
L_\ell(s, \ad(f)) = \left[ \left( 1 - \frac{\alpha_\ell}{\beta_\ell} \ell^{-s} \right) (1-\ell^{-s}) \left( 1 - \frac{\beta_\ell}{\alpha_\ell} \ell^{-s} \right) \right]^{-1}
\]
at the unramified places $\ell$, and $\alpha_\ell$ and $\beta_\ell$ are the roots of the $\ell$-th Hecke polynomial $X^2 - a_\ell(f) X + \chi(\ell) \ell^{k-1}$. By the work of Shimura \cite{S75} (and \cite{GJ} for general automorphic representations of $\GL_2$), $L(s, \ad(f))$ admits meromorphic continuation to $s \in \C$ and satisfies a functional equation under $s \leftrightarrow 1-s$. In \cite{H81a}, Hida established an integrality result for the algebraic part $L^\alg(1, \ad(f))$ (i.e., $L(1, \ad(f))$ divided by the product of the Manin periods), and shows that the prime factors of $L^\alg(1, \ad(f))$ are congruence primes of $f$, i.e., primes $p$ for which there exists another eigenform $g$ such that $f \equiv g \pmod{p}$. The converse was proved in \cite{H81b} for primes which are ordinary for $f$, and in \cite{R83} in general.

In the ordinary case, Hida's criterion is deduced from a precise identity relating the adjoint $L$-value $L(1, \ad(f))$ with the size of the congruence module associated with $f$ \cite{H81a, H88b}, which allows him to construct an \emph{algebraic} $p$-adic $L$-function interpolating these values as $f$ varies in an ordinary family. The congruence module is in turn related to the Selmer group of the adjoint motive of $f$ by Wiles' proof of Fermat's last theorem. This connection be understood as a non-abelian class number formula, or more precisely the Bloch--Kato conjecture for the adjoint motive of $f$.

In his PhD thesis \cite{U95}, Urban generalized Hida's results to the case of Bianchi modular forms $F$ over an imaginary quadratic field $K$; namely, he shows that $L(1, \ad(F))$ modulo suitable periods $u_1(F)$, $u_2(F)$, which are defined via comparison between de Rham and Betti cohomologies of the associated Bianchi threefold, is an algebraic integer and gives the size of the congruence module of $F$; hence, its prime factors are congruence primes of $F$. When $F = f_K$ is the base-change of a classical eigenform $f$ to $K$, there is a factorization
\[
L(s, \ad(F)) = L(s, \ad(f)) L(s, \ad(f) \otimes \alpha),
\]
where $\alpha$ is the quadratic character associated with $K$ by class field theory.  By combining the results of Hida and Urban, the value $L(1, \ad(f) \otimes \alpha)$ thus controls congruences between the base-change of $f$ and non-base-change Bianchi forms on $K$. It is worth noting that $L(1, \ad(f) \otimes \alpha)$ is \emph{non-critical} in the sense of Deligne.

The value $L(1, \ad(f) \otimes \alpha)$ was further studied by Hida \cite{H99}, where an integral formula is proved despite its non-criticality. From this, Hida establishes an integrality result for $L(1, \ad(f) \otimes \alpha)$ modulo the period $u_2(f_K)$ and formulates a conjecture relating this value with, in essence, the Selmer group of the adjoint motive of $f$ twisted by $\alpha$. In a recent work \cite{TU18}, Tilouine and Urban established integral period relations of base-change modular forms and prove a version of Hida's conjecture.

It is natural to interpolate the result of \cite{TU18} in a $p$-adic family and formulate an analogue of the Iwasawa--Greenberg main conjecture for $\ad(f) \otimes \alpha$ in the weight direction. As a first step towards this, our goal is to construct an \emph{analytic} $p$-adic $L$-function interpolating the values $L(1, \ad(f) \otimes \alpha)$ as $f$ varies in a $p$-adic family.

\subsection{Main result}

As before, let $K$ be an imaginary quadratic field with corresponding quadratic character $\alpha$, and $p$ be an odd prime which is split in $K$. Fix a sufficiently large $p$-adic ring of integers $\O$, and denote by $\Lambda_\Q := \O[[1+p\Z_p]]$ the Iwasawa algebra. Let $\omega: (\Z/p\Z)^\times \to \Z_p^\times$ be the $p$-adic Teichm\"uller character.

In this paper, we construct an analytic $p$-adic $L$-function that interpolates the algebraic part of $L(1, \ad(f) \otimes \alpha)$ as $f$ varies in a Hida family. For simplicity, we confine our attention to the simplest setting when the eigenforms $f$ involved have Nebentype $\alpha$ (and necessarily odd weight); a more general construction following similar ideas will be carried out in a future work.

To this end, consider a Hida family $\lambda: \mb{h}_\Q \to \mb{I}$ with tame character $\alpha \omega^r$ for some integer $r$ (necessarily odd); here $\mb{h}_\Q$ is the universal ordinary Hecke algebra for $\GL_2(\Q)$ and $\mb{I}$ is a finite flat extension of $\Lambda_\Q$ (see \cite{H86a, H86b}). For further simplicity, let us state the main result in the case $\mb{I} = \Lambda_\Q$. Thus $\lambda$ gives an ordinary $\Lambda_\Q$-adic form $\mb{f} \in \mb{S}^\ord(N, \alpha \omega^r; \Lambda_\Q)$ which specializes to an ordinary eigenform $\mb{f}_n \in S_n(Np, \alpha \omega^{r-n})$ at each $n \in \Z_{\geq 2}$; in particular, when $n \equiv r \pmod{p-1}$, $\mb{f}_n \in S_n(Np, \alpha)$ has Nebentype $\alpha$. Denote
\[
A_r := \{n \in \Z_{\geq 2}: n \equiv r \pmod{p-1}\},
\]
whch is a dense subset of $\Z_p$.

\begin{main}
Suppose $\mb{f} \in \mb{S}^\ord(N, \alpha \omega^r; \Lambda_\Q)$ is a Hida family. Then there exists $\mc{L} \in \Frac(\Lambda_\Q)$ which specializes to
\[
c_n \cdot (1 - a_p(\mb{f}_n)^{-2} p^{n-1}) L^\alg(1, \ad(\mb{f}_n) \otimes \alpha)
\]
at almost all $n \in A_r$. Here the algebraic part $L^\alg$ is defined in Proposition \ref{alg-part}, and the $p$-adic error term $c_n \in \overline{\Q_p}^\times$ is defined in Definition \ref{error}.
\end{main}

To the best of our knowledge, this is the first instance of a $p$-adic $L$-function that interpolates non-critical $L$-values. A precise statement, which covers the case of $\mb{I}$-adic forms (still with tame character $\alpha \omega^r$), is given in Theorem \ref{main-thm}.

\begin{rem}
\leavevmode
\begin{enumerate}
\item We make no attempt to control the $p$-adic error terms $c_n$ -- see the discussion following Definition \ref{error} -- but simply remark that they are of a similar nature to those appearing in the $p$-adic $L$-function of \cite{GS93}. In the Bianchi setting, the three-variable $p$-adic $L$-function constructed in \cite{BW-Bianchi} involves $p$-adic error terms defined using $H^1$ of the Bianchi threefold, whereas our $p$-adic error terms are defined using $H^2$.

\item For a Hida family with general tame character, a similar construction should yield a $p$-adic $L$-function whose specializations on (a translate of) $A_r$ can be determined. Although $A_r$ is a dense subset of $\Z_p$, it would be more satisfactory to relax this condition. The difficulty is caused by the form of Hida's integral formula for $L(1, \ad(f) \otimes \alpha)$, and will be briefly addressed in the remark following Theorem \ref{main-thm}.
\end{enumerate}
\end{rem}

\subsection{Outline}

We give an overview of the strategy and this paper.

The starting point is Hida's integral formula for $L(1, \ad(f) \otimes \alpha)$ in \cite{H99}, which is reviewed in Section \ref{2} following a terse account on Bianchi modular forms. Denote by $f_K$ the base-change Bianchi form of $f$. For $q=1,2$, the Eichler--Shimura--Harder isomorphism associates to $f_K$ a cohomology class
\[
\delta_q(f_K) \in H_\cusp^q(Y_K, V_{n,n}(\C)),
\]
which can be represented by a harmonic $q$-form on the Bianchi threefold $Y_K$; here $V_{n,n}(\C)$ is the local system corresponding to the irreducible polynomial representation $\Sym^n \C^2 \otimes \overline{\Sym^n \C^2}$ of $\GL_2(K)$.

Hida's formula expresses the value $L(1, \ad(f) \otimes \alpha)$ as a suitable integral of $\delta_2(f_K)$ over the modular curve $Y_\Q$ viewed as a $2$-cycle in $Y_K$. This can be summarized as a linear form
\[
\mathbb{L}_n: H_c^2(Y_K, V_{n,n}(\C)) \to H_c^2(Y_\Q, V_{n,n}(\C)) \to H_c^2(Y_\Q, \C) \overset{\sim}{\to} \C,
\]
where the first map is induced by restriction along $Y_\Q \hookrightarrow Y_K$, the second map is induced by the canonical projection $V_{n,n}(\C) \to \C$ of $\SL_2(\Z)$-representations under the Clebsch--Gordan decomposition, and the third map is integration over $Y_\Q$. A key feature of the linear form $\mathbb{L}_n$ is that it makes sense over a number field in place of $\C$ (hence over a $p$-adic field). Upon normalizing $\delta_2(f_K)$ by Urban's period $u_2(f_K)$, we obtain a cohomology class $\widehat{\delta}_2(f_K) \in H_c^2(Y_K, V_{n,n}(\O))$ defined over some valuation ring $\O$, and the algebraic part of $L(1, \ad(f) \otimes \alpha)$:
\[
L^\alg(1, \ad(f) \otimes \alpha) = \mathbb{L}_n \left( \widehat{\delta}_2(f_K) \right) \in \O.
\]

In Section \ref{3} we begin working $p$-adically, replacing the coefficient ring $\O$ by a $p$-adic completion. To construct a $p$-adic $L$-function for $L^\alg(1, \ad(f) \otimes \alpha)$, we adopt the standard technique of $\Lambda$-adic modular symbols following Mazur--Kitagawa \cite{K94} and Greenberg--Stevens \cite{GS93}. The problem essentially boils down to the following steps:
\begin{enumerate}[label=(\alph*)]
\item constructing an Iwasawa module $\mb{M}$ that specializes to $H_c^2(Y_K, V_{k,k}(\O))$ for each $k$;
\item defining a linear form on $\mb{M}$ that interpolates $\mathbb{L}_k: H_c^2(Y_K, V_{k,k}(\O)) \to \O$ as $k$ varies $p$-adically over the weight space;
\item extracting an element in (the ordinary part of) $\mb{M}$ that belongs to a system of Hecke eigenvalues corresponding to a given Hida family.
\end{enumerate}

Following \cite{GS93}, $\mb{M}$ will be the cohomology group $H_c^2(Y_K, \mc{D})$ with coefficients in an appropriate space of measures $\mc{D}$, which is a module over the Iwasawa algebra $\Lambda = \O[[\Z_p^\times]]$. The key properties are that:
\begin{itemize}
\item $\mc{D}$ admits natural specialization maps $\mc{D} \to V_{k,k}(\O)$ for all $k$ (see Section \ref{section-specialization});
\item the Clebsch--Gordan projection $V_{k,k}(\O) \to \O$ involved in Hida's evaluation map $\mathbb{L}_k$ is the specialization of a map $\mc{D} \to \Lambda$ (see Section \ref{section-evaluation}).
\end{itemize}
Then one might expect the induced diagram
\[
\xymatrix{
H_c^2(Y_K, \mc{D}) \ar[r]^-{\mathbb{L}_\Lambda} \ar[d] & \Lambda \ar[d] \\
H_c^2(Y_K, V_{k,k}(\O)) \ar[r]^-{\mathbb{L}_k} & \O
}
\]
to provide an answer to (a) and (b). Unfortunately this diagram does not commute, but the defect can be explicitly computed as an Euler factor; this cohomological calculation is carried out in Section \ref{4}.

In Section \ref{5}, we address (c) by applying Hida theory to find a desired eigensystem $\mc{F}$ in $H_c^2(Y_K, \mc{D})$. Comparing its specializations with the normalized classes $\widehat{\delta}_2(-) \in H_c^2(Y_K, V_{k,k}(\O))$ gives the $p$-adic error terms $c_k$. 

\begin{rem}
It is best to refrain from calling these \emph{$p$-adic periods} because they do not arise from the comparison between cohomological realizations of the underlying motive.
\end{rem}

Finally, the $p$-adic $L$-function is constructed by evaluating $\mathbb{L}_\Lambda$ at the class $\mc{F}$.

\subsection{Further directions}

As previously explained, one of the motivations for constructing an analytic $p$-adic $L$-function for $L(1, \ad(f) \otimes \alpha)$ is to study a family version of the results of \cite{TU18}. To formulate an analogue of the Iwasawa--Greenberg main conjecture in this setting (where the $L$-value is non-critical), it is necessary to clarify the nature of the $p$-adic error terms, which seems difficult at the moment.

In another direction, it should be evident to the reader that the construction in this paper can be readily adapted to the setting of overcohomology cohomology and yield a $p$-adic $L$-function $L_p(1, \ad(f) \otimes \alpha)$ on the eigencurve (i.e., on Coleman families of finite-slope eigenforms), albeit only over a neighborhood of the weight space. Similarly to how the zeros of the $p$-adic adjoint $L$-function $L_p(1, \ad(f))$ are related to the ramification locus of the eigencurve \cite{Kim, Bellaiche}, we expect the zeros of $L_p(1, \ad(f) \otimes \alpha)$ to encode the intersection locus between the base-change and non-base-change components in the imaginary quadratic eigencurve. We hope to return to this question in the future.

\section{Preliminaries}

\label{2}

We begin by giving a brief account of Bianchi modular forms, i.e., automorphic forms on $\GL_2(K)$. Rather than giving a self-contained survey, our focus is merely on setting up notation; therefore we will omit details that are not necessary for the remainder of this paper, and refer the reader to other sources, such as \cite{H94-L}, \cite{U95}, \cite{G99} and \cite{W17}.

In the second half of this section, we recall Hida's integral formula for $L(1, \ad(f) \otimes \alpha)$ in \cite{H99} and interpret it in terms of cohomology.

\subsection{Bianchi modular forms}

Let $K = \Q(\sqrt{-D})$ be an imaginary quadratic field with discriminant $-D$ and different $\mf{d} = (\sqrt{-D})$. Let $\alpha = \left( \frac{-D}{} \right)$ be the quadratic character associated with $K$, which is odd.

Let $\O_K$ be the ring of integers, $\widehat{\O}_K = \O_K \otimes_\Z \widehat{\Z}$, and $\A_K$ (resp.\ $\A_K^\times$, $\A_{K,f}$, $\A_{K,f}^\times$) be the ring of adeles (resp.\ ideles, finite adeles, finite ideles) over $K$.

For any commutative ring $R$, let $W_n(R)$ be the space of homogeneous polynomials in two variables (denoted $S$ and $T$) of degree $n$ with coefficients in $R$, equipped with the natural right action of $\GL_2$: for $k = \begin{pmatrix} a & b \\ c & d \end{pmatrix} \in \GL_2(R)$ and $P \in W_n(R)$, we have
\[
(P|k)(S,T) = P((S, T) \cdot {}^t k),
\]
i.e., $(P|k)(S, T) = P(aS+bT, cS+dT)$. In particular, this restricts to a right action of $\SU_2(\C)$ on $W_n(\C)$.

\begin{rem}
\leavevmode
\begin{itemize}
\item Later we will use $V_n(R)$ to denote the same space, but considered with a left action of $\GL_2$.
\item For consistency with literature, we use the slightly misleading notation $\SU_2(\C) \subset \GL_2(\C)$ to denote the special unitary group, despite it being the real points of a group scheme.
\end{itemize}
\end{rem}

\subsubsection{Adelic automorphic forms}

\begin{defn}
Let $U$ be a compact open subgroup of $\GL_2(\A_{K,f})$. Denote by $S_{n,n}(U)$ the space of cuspidal automorphic forms on $\GL_2(K)$ of parallel weight $(n,n)$ and level $U$, which are functions $F: \GL_2(\A_K) \to W_{2n+2}(\C)$ satisfying:
\begin{enumerate}
\item $F(\gamma g) = F(g)$ for all $\gamma \in \GL_2(K)$;
\item $F(gu) = F(g)$ for all $u \in U$;
\item $F(zgk) = |z|^{-n} \cdot F(g)|k$ for all $z \in Z(\GL_2(\C)) \cong \C^\times$ and $k \in \SU_2(\C)$;
\item a harmonicity condition;
\item a cuspidality condition.
\end{enumerate}
\end{defn}

\begin{rem}
More generally, cohomological weights for $\GL_2(K)$ are paramatrized by pairs $(\mb{n}, \mb{v})$, where $\mb{n} = (n,n) \in \Z^2$ with $n \geq 0$ and $\mb{v} = (v_1, v_2) \in \Z^2$ is arbitrary, but the special case $\mb{v} = (0,0)$ is sufficient for our setting.
\end{rem}

In this paper, we will restrict our attention to $S_{n,n}(U_0(N))$, where
\[
U_0(N) = \left\{ \begin{pmatrix} a & b \\ c & d \end{pmatrix} \in \GL_2(\widehat{\O}_K): c \in N \widehat{\O}_K \right\}
\]
for some positive integer $N$; in other words, the Bianchi cusp forms we consider will always have level $N$ and \emph{trivial} Nebentype.

Every $F \in S_{n,n}(U_0(N))$ has a Fourier--Whittaker expansion
\begin{equation}
\label{Fourier-adelic}
F \begin{pmatrix} y & x \\ 0 & 1 \end{pmatrix} = |y|_{\A_K} \sum_{\xi \in K^\times} a(\xi y \mf{d}; F) W(\xi y_\infty) \mb{e}_K(\xi x)
\end{equation}
where:
\begin{itemize}
\item $|\cdot|_{\A_K}$ is the usual idele character on $\A_K^\times$ trivial on $K^\times$;
\item $\mf{d} = (\sqrt{-D})$ is the different of $K/\Q$;
\item $\mf{n} \mapsto a(\mf{n}; F)$ is a $\C$-valued function on the fractional ideals of $K$ which vanishes outside the set of integral ideals;
\item $W: \C^\times \to W_{2n+2}(\C)$ is the Whittaker function
\[
W(s) = \sum_{\alpha = 0}^{2n+2} \binom{2n+2}{\alpha} \left( \frac{s}{i|s|} \right)^{n+1-\alpha} K_{\alpha-n-1}(4\pi|s|) S^{2n+2 - \alpha} T^{\alpha},
\]
where $K_\alpha$ is the modified Bessel function;
\item $\mb{e}_K$ is the standard additive character of $\mb{A}_K$ trivial on $K$, given by
\[
\mb{e}_K = (\mb{e}_\infty \circ \tr_{\C/\R}) \cdot \prod_\mf{p} (\mb{e}_p \circ \tr_{K_\mf{p}/\Q_p}),
\]
where $\mb{e}_\infty(r) = e^{2\pi ir}$ and $\mb{e}_p(\sum_j c_j p^j) = e^{-2\pi i \sum_{j<0} c_j p^j}$.
\end{itemize}

The Hecke algebra $h_{n,n}(U_0(N)) \subset \End_\C(S_{n,n}(U_0(N))$ acts on the space $S_{n,n}(U_0(N))$, but we will not recall the precise formulas here. To every Hecke eigenform $F$, i.e., a cusp form which is an eigenfunction under all the Hecke operators $T_\mf{n}$, one can associate an algebra homomorphism $\lambda_F: h_{n,n}(U_0(N)) \to \C$ such that $T_\mf{n} F = \lambda_F(T_\mf{n}) F$. An eigenform $F$ is said to be normalized if $a(\O_K; F) = 1$ in the Fourier--Whittaker expansion (\ref{Fourier-adelic}).

\subsubsection{Classical automorphic forms}

Next we review how an adelic Bianchi modular form on $\GL_2(K)$ gives rise to a tuple of classical automorphic forms on the upper half-space
\[
\mc{H}_3 = \{x+iy: x \in \C, y \in \R_{>0}\},
\]
which can be identified with the symmetric space $\GL_2(\C) / (\C^\times \cdot \SU_2(\C))$.

Let $h$ be the class number of $K$, and fix a set of representatives $a_i \in \A_{K,f}^\times$ of the class group (with $a_1 = 1$). Then the strong approximation theorem for $\GL_2$ gives
\[
\GL_2(\A_K) = \coprod_{i=1}^h \GL_2(K) \cdot t_i U_0(N) \GL_2(\C),
\]
where $t_i = \begin{pmatrix} a_i & 0 \\ 0 & 1 \end{pmatrix}$. We set
\[
\Gamma_0^{K,i}(N) = \SL_2(K) \cap t_i U_0(N) \GL_2(\C) t_i^{-1} \subset \SL_2(K)
\]
and often abbreviate this as $\Gamma^i$. For $i=1$, note that $\Gamma_0^{K,1}(N) = \left\{ \begin{pmatrix} a & b \\ c & d \end{pmatrix} \in \GL_2(\O_K) : N \mid c \right\}$ is the congruence subgroup $\Gamma_0^K(N)$ of level $N$ in the classical sense.

The (adelic) Bianchi threefold of level $N$ is the locally symmetric space
\[
\widetilde{Y}_K(N) = \GL_2(K) \backslash \GL_2(\A_K) / (\C^\times \cdot \SU_2(\C) \cdot U_0(N)).
\]
It decomposes into connected components
\[
\widetilde{Y}_K(N) = \coprod_{i=1}^h Y_K^i(N),
\]
where
\begin{align*}
Y_K^i(N) &= \Gamma_0^{K,i}(N) \backslash \GL_2(\C) / (\C^\times \cdot \SU_2(\C)) \\
&= \Gamma_0^{K,i}(N) \backslash \mc{H}_3.
\end{align*}

Let $F \in S_{n,n}(U_0(N))$. For each $1 \leq i \leq h$, the function $F_i: \GL_2(\C) \to W_{2n+2}(\C)$ defined by
\[
F_i(g) = F(t_i g)
\]
descends to a classical automorphic form $\mc{H}_3 \to W_{2n+2}(\C)$ of weight $(n,n)$ and level $\Gamma_0^{K,i}$; this space is denoted $S_{n,n}(\Gamma_0^{K,i}(N))$. It is easily checked that the Fourier--Whittaker expansion (\ref{Fourier-adelic}) descends to
\begin{align*}
F_i(x+iy) = |a_i|_{\A_K} y \sum_{\alpha=0}^{2n+2} & \binom{2n+2}{\alpha} \\
& \cdot \left[ \sum_{\xi \in K^\times} a(\xi a_i \mf{d}) \left( \frac{\xi}{i|\xi|} \right)^{n+1-\alpha} K_{\alpha-n-1}(4 \pi y |\xi|) \mb{e}_K(\xi x) \right] S^{2n+2-\alpha} T^\alpha
\end{align*}
for $x+iy \in \mc{H}_3$. In this manner, every cuspidal automorphic form on $\GL_2(\A_K)$ determines an $h$-tuple of automorphic forms on $\mc{H}_3$.

Our focus will be on the component for $i=1$, so we will simplify notation as $Y_K(N) = Y_K^1(N)$ and call this the (classical) Bianchi threefold of level $N$. The classical automorphic form $F_1: \mc{H}_3 \to W_{2n+2}(\C)$ belongs to $S_{n,n}(\Gamma_0^K(N))$.

\subsubsection{Base change}

Now we state the basic properties of base-change lifts from $\GL_2(\Q)$ to $\GL_2(K)$. Details can be found in \cite{Jacquet-2} and \cite{Langlands-BC}. For a classical elliptic newform $f \in S_k(N, \psi)$, we denote its base-change adelic Bianchi form by $f_K$ or $\BC(f)$, preferring the latter notation when other subscripts are present. Throughout we impose:

\begin{hyp}
$f_K$ is cuspidal.
\end{hyp}

This is satisfied, for example, when $f$ is not a CM form with character $\alpha$ (the quadratic character associated with $K$). Then:
\begin{itemize}
\item $f_K$ is a cuspidal Bianchi newform of weight $(k-2,k-2)$, level $N$ and Nebentype $\chi = \psi \circ N_{K/\Q}$ (see the next remark).
\item The Hecke eigenvalues of $f_K$ can be described in terms of the eigenvalues $a_\ell$ of $f$: for every prime $\mf{l}$ of $K$ above $\ell$,
\[
a_\mf{l} = 
\begin{cases}
a_\ell &\text{if $\ell$ is split or ramified}, \\
a_\ell^2 - 2 \psi(\ell) \ell^{k-1} &\text{if $\ell$ is inert},
\end{cases}
\]
where $\psi(\ell)$ is understood as $0$ if $\ell$ divides the level $N$.
\end{itemize}

\begin{rem}
Our negligence to define the Nebentype of a Bianchi form is going to be harmless, as all base-change forms considered in this paper have trivial Nebentype and hence belong to $S_{k-2,k-2}(U_0(N))$.
\end{rem}

\subsection{Cohomology and the Eichler--Shimura--Harder isomorphism}

\label{section-cohomology}

The Eichler--Shimura--Harder isomorphism relates spaces of (cohomological) automorphic forms and cohomology groups of the associated locally symmetric spaces. For the purpose of introducing Hida's integral formula, it is enough to consider the classical version of this isomorphism, i.e., between classical Bianchi forms and cohomology of the classical Bianchi threefold $Y_K^1(N)$, rather than the adelic counterpart $\widetilde{Y}_K(N)$. This is sufficient because we will only deal with Bianchi modular forms arising from base change of elliptic modular forms.

\subsubsection{Cohomology of the Bianchi threefold}

For the rest of the paper, we impose:

\begin{hyp}
$Y_K(N)$ is smooth, and $\Gamma_0^K(N)/(\Gamma_0^K(N) \cap K^\times)$ is torsion-free.
\end{hyp}

These are satisfied for $N$ sufficiently large. 

If $M$ is any $\Gamma_0^K(N)$-module, we can consider the locally constant sheaf of continuous sections of 
\[
\Gamma_0^K(N) \backslash \left( \GL_2(\C) / (\C^\times \cdot \SU_2(\C)) \times M \right) \to \Gamma_0^K(N) \backslash \GL_2(\C) / (\C^\times \cdot \SU_2(\C)) = Y_K(N)
\]
where $M$ is equipped with the discrete topology, and hence the cohomology groups
\[
H^q(Y_K(N), M), \quad H_c^q(Y_K(N), M),
\]
as well as the cuspidal cohomology group
\[
H_\cusp^q(Y_K(N), M) := \mathrm{im} \left(H_c^q(Y_K(N), M) \to H^q(Y_K(N), M) \right).
\]
Under our hypotheses, $Y_K(N)$ is an Eilenberg--MacLane space for $\Gamma_0^K(N)$, so $H^q(Y_K(N), M)$ can be identified with group cohomology $H^q(\Gamma_0^K(N), M)$.

If, furthermore, $M$ has the structure of a module over the monoid
\[
\Delta_0^K(N) = \left\{ \begin{pmatrix} a & b \\ c & d \end{pmatrix} \in M_2(\O_K) \cap \GL_2(K): N \mid c, N \nmid d \right\},
\]
then the cohomology groups $H_*^q(Y_K(N), M)$ are equipped with actions of Hecke operators at principal ideals of $\O_K$; we will not recall the definition.

\begin{rem}
In general, Hecke operators at non-principal ideals must be defined adelically on the cohomology of $\widetilde{Y}_K(N)$. However, this will not matter for the rest of the paper, so we gloss over this point.
\end{rem}

\subsubsection{Cohomology of the modular curve}
The situation over $\Q$ is completely analogous. Consider the congruence subgroup
\[
\Gamma_0^\Q(N) = \left\{ \begin{pmatrix} a & b \\ c & d \end{pmatrix} \in \SL_2(\Z): N \mid c \right\} = \Gamma_0^K(N) \cap \SL_2(\Q)
\]
and the modular curve
\[
Y_\Q(N) = \Gamma_0^\Q(N) \backslash \mc{H}_2.
\]
Every $\Gamma_0^\Q(N)$-module $M$ gives rise to a local system on $Y_\Q(N)$, and hence cohomology groups $H_*^q(Y_\Q(N), M)$.

If $M$ is in addition a module over the monoid
\[
\Delta_0^\Q(N) = \left\{ \begin{pmatrix} a & b \\ c & d \end{pmatrix} \in M_2(\Z) \cap \GL_2(\Q): N \mid c, (d,N) = 1 \right\},
\]
then $H_*^q(Y_\Q(N), M)$ has an action of Hecke operators.

\subsubsection{Eichler--Shimura isomorphism}

For any commutative ring $R$, consider the space $V_n(R) = \Sym^n(R^2)$ of homogeneous polynomials in $(X,Y)$ with coefficients in $R$, equipped with a left action of $\GL_2(R)$ by
\[
(\gamma \cdot P)(X,Y) = P((X,Y) \cdot {}^t \gamma^\iota),
\]
where $\gamma^\iota = \det(\gamma) \gamma^{-1}$, i.e.,
\[
\begin{pmatrix} a & b \\ c & d \end{pmatrix} \cdot P(X,Y) = P(dX-bY, -cX+aY).
\]

For $(n_1, n_2) \in \Z_{\geq 0}^2$, define
\[
V_{n_1,n_2}(R) = \Sym^{n_1}(R^2) \otimes \Sym^{n_2}(R^2),
\]
which is likewise identified concretely as the space of polynomials in $(X_i, Y_i)_{i=1,2}$ with coefficients in $R$ homogeneous of degree $n$ for each pair $(X_\sigma, Y_\sigma)$. This space is equipped with a left action of $\GL_2(R)^2$ by the formulas above.

If $R$ is a subalgebra of $\C$, this restricts to an action of $\GL_2(R) \subset \GL_2(\C)^2$, embedding into the first component via identity and the second component via complex conjugation.

Applying this discussion to $M = V_{k,k}(R)$, viewed as a module over $\Gamma_0^K(N)$, we have the cohomology groups
\[
H^q(Y_K(N), V_{n,n}(R)), \quad H_c^q(Y_K(N), V_{n,n}(R)), \quad H_\cusp^q(Y_K(N), V_{n,n}(R)).
\]
They are equipped with actions of Hecke operators at principal ideals of $\O_K$.

We are now ready to state:

\begin{thm}[Eichler--Shimura, Harder]
For $q = 1, 2$, there is a Hecke-equivariant isomorphism
\[
\delta_q: S_{n,n}(\Gamma_0^K(N)) \overset{\sim}{\to} H_\cusp^q(Y_K(N), V_{n,n}(\C)).
\]
\end{thm}

We describe this isomorphism in the case $q=2$ explicitly, following \cite{H94-L}. With no risk of ambiguity, the map $\delta_2$ will often be denoted $\delta$.

Let $F \in S_{n,n}(\Gamma_0^K(N))$ be a (classical) Bianchi form, with Fourier--Whittaker expansion
\begin{align*}
F(x+iy) = y^{n+2} \sum_{\xi \in K^\times} & a(\xi y \mf{d}; F) \\
& \cdot \left[ \sum_{\alpha = 0}^{2n+2} \binom{2n+2}{\alpha} \left( \frac{\xi}{i|\xi|} \right)^{n+1-\alpha} K_{\alpha-n-1}(4\pi|\xi|y) S^{2n+2-\alpha} T^\alpha \right] e^{2\pi i(\xi x + \overline{\xi x})}
\end{align*}
for $x+iy \in \mc{H}_3$. Following Hida's recipe, we group the coefficients as
\[
F = \sum_{\alpha=0}^{2n+2} G_\alpha \binom{2n+2}{\alpha} S^{2n+2-\alpha} T^\alpha
\]
and expand the formal identity
\begin{align*}
& (X_1 V - Y_1 U)^n (X_2 U + Y_2 V)^n (AV - BU)^2 \\
=& \left( \sum_{j_1=0}^n (-1)^{j_1} \binom{n}{j_1} X_1^{n-j_1} Y_1^{j_1} V^{n-j_1} U^{j_1} \right) \cdot \left( \sum_{j_2=0}^n \binom{n}{j_2} X_2^{n-j_2} Y_2^{j_2} U^{n-j_2} V^{j_2} \right) \\
&\phantom{= \sum} \cdot (A^2V^2 - 2ABUV + B^2U^2) \\
=& \sum_{0 \leq \mb{j} \leq \mb{n}} (-1)^{j_1} \binom{\mb{n}}{\mb{j}} X^{\mb{n}-\mb{j}} Y^\mb{j} (U^{n+j_1-j_2} V^{n-j_1+j_2+2} A^2 \\
&\phantom{= \sum} - 2 U^{n+j_1-j_2+1} V^{n-j_1+j_2+1} AB + U^{n+j_1-j_2+2} V^{n-j_1+j_2} B^2),
\end{align*}
where the multi-index notation $\binom{\mb{n}}{\mb{j}} = \binom{n_1}{j_1} \binom{n_2}{j_2}$ and $X^\mb{a} = X_1^{a_1} X_2^{a_2}$ are used. Finally, we make the following substitutions:
\begin{itemize}
\item $U^\alpha V^{2n+2-\alpha}$ by $(-1)^{2n+2-\alpha} G_\alpha$,
\item $(A,B)$ by $(y^{-1/2} A, y^{-1/2} B)$,
\item $(A^2, AB, B^2)$ by $y^{-1} (dy \wedge dx, -2 dx \wedge d\overline{x}, dy \wedge d\overline{x})$.
\end{itemize}

\begin{prop}[\cite{H94-L}]
\label{explicit-ES}
The image of $F \in S_{n,n}(\Gamma_0^K(N))$ under the Eichler--Shimura isomorphism $\delta$ is the harmonic differential form
\begin{align*}
\delta(F) =& w \cdot \sum_{0 \leq \mb{j} \leq \mb{n}} (-1)^{n-j_2} \binom{\mb{n}}{\mb{j}} X^{\mb{n}-\mb{j}} Y^\mb{j} (G_{n+j_1-j_2} y^{-2} dy \wedge dx \\
&\phantom{= \sum} - G_{n+j_1-j_2+1} y^{-2} dx \wedge d\overline{x} + G_{n+j_1-j_2+2} y^{-2} dy \wedge d\overline{x}),
\end{align*}
on $\mc{H}_3$, where $w = \begin{pmatrix} y^{1/2} & x y^{-1/2} \\ 0 & y^{-1/2} \end{pmatrix} \in \SL_2(\C)$ acts on the variables $(X_1, Y_1, X_2, Y_2)$ by
\[
w \cdot (X_1, Y_1, X_2, Y_2) = (X_1, Y_1, X_2, Y_2) \begin{pmatrix} {}^t w^\iota & 0 \\ 0 & {}^t \overline{w}^\iota \end{pmatrix}.
\]
\end{prop}

\subsection{\texorpdfstring{$L$}{L}-functions}

We introduce the adjoint $L$-function for classical eigenforms and the Asai $L$-function for Bianchi eigenforms, and state a relationship between them in the base-change situation. Further detail and motivic interpretation can be found in \cite{G99} and \cite{H99}.

Throughout, $\eta: (\Z/m\Z)^\times \to \C^\times$ will denote a primitive Dirichlet character. For $L$-functions with an Euler product expansion, $L_N(s, -)$ denotes the same $L$-function with Euler factors at primes dividing $N$ removed.

Let $f = \sum_{n=1}^\infty a_n(f) q^n \in S_k(N, \psi)$ be a normalized eigenform. For each prime $p \nmid N$, suppose the $p$-th Hecke polynomial $X^2 - a_p(f) X + \psi(p) p^{k-1}$ has roots $\alpha_p$ and $\beta_p$.

\begin{defn}[Adjoint $L$-function]
The (twisted) adjoint $L$-function of $f$ is
\[
L(s, \ad(f) \otimes \eta) := \prod_{p \nmid N} \left[ \left( 1 - \frac{\alpha_p}{\beta_p} \cdot \eta(p) p^{-s} \right) (1 - \eta(p) p^{-s}) \left( 1 - \frac{\beta_p}{\alpha_p} \cdot \eta(p) p^{-s} \right) \right]^{-1}.
\]
\end{defn}

It has meromorphic continuation to $s \in \C$, and satisfies a functional equation under $s \leftrightarrow 1-s$.

Now let $F \in S_{n,n}(U_1(N))$ be a normalized Bianchi eigenform with Nebentype $\chi: (\O_K/N\O_K)^\times \to \C^\times$. Denote by $\chi_\Q: (\Z/N\Z)^\times \to \C^\times$ the restriction of $\chi$. Recall the Fourier--Whittaker coefficients $a(\mf{n}; F)$ from (\ref{Fourier-adelic}).

\begin{defn}[Asai $L$-function]
The (twisted) Asai $L$-function of $F$ is
\[
L(s, \As(F) \otimes \eta) := L_{mN}(2s-2k-2, \eta^2 \chi_\Q) \cdot L_{K/\Q}(s-1, F, \eta),
\]
where
\[
L_{K,\Q}(s, F, \eta) := \sum_{n=1}^\infty \frac{\eta(n) a(n\O_K; F)}{n^{s+1}}.
\]
\end{defn}

\begin{rem}
Again, we will only be concerned with Bianchi forms with trivial Nebentype.
\end{rem}

The Asai $L$-function is absolutely convergent for $\mathrm{Re}(s)$ sufficiently large, and has meromorphic continuation to $s \in \C$.

\begin{lem}
\label{factorization}
If $F = f_K$ is the base-change of $f \in S_k(N, \psi)$, then there is a factorization
\[
L(s+k-1, \As(F) \otimes \eta) = L(s, \alpha \psi \eta) L(s, \ad(f) \otimes \psi \eta)
\]
up to Euler factors at primes dividing $mN$.
\end{lem}

\begin{proof}[Proof (sketch)]
This follows by comparing the Euler factors on both sides. The Euler product expansion of the Asai $L$-function can be found in \cite{G99}.
\end{proof}

\subsection{Special value of the twisted adjoint \texorpdfstring{$L$}{L}-function}

This subsection is an exposition of Hida's integral expression for the special value $L(1, \ad(f) \otimes \alpha)$. For clarity, we only set up the calculation for a classical eigenform with Nebentype $\alpha$, so that the base-change Bianchi form has trivial Nebentype.

More precisely, let $f \in S_{n+2}(N, \alpha)$ be a normalized eigenform with Nebentype $\alpha$. In particular, this forces $n$ to be odd and $N$ to be divisible by $D$. The base-change adelic Bianchi eigenform $F = f_K$ belongs to $S_{n,n}(U_0(N))$, and projects to the classical Bianchi eigenform $F_1 \in S_{n,n}(\Gamma_0^K(N))$, which is a function on $\mc{H}_3$.

To furthur simplify notation, we denote the (classical) Bianchi threefold and modular curve of level $N$ as
\begin{align*}
Y_K &:= Y_K(N) = \Gamma_0^K(N) \backslash \mc{H}_3, \\
Y_\Q &:= Y_\Q(N) = \Gamma_0^\Q(N) \backslash \mc{H}_2;
\end{align*}
this will not cause any confusion as $N$ will not change. The natural embedding $\mc{H}_2 \hookrightarrow \mc{H}_3$ induces $Y_\Q \hookrightarrow Y_K$. For suitable local systems $M$, we will consider the cohomology groups $H_*^q(Y_K, M)$ and $H_*^q(Y_\Q, M)$.

\subsubsection{Restriction to $\Q$}

Recall the Eichler--Shimura isomorphism
\[
\delta: S_{n,n}(\Gamma_0^K(N)) \overset{\sim}{\to} H_\cusp^2(Y_K, V_{n,n}(\C))
\]
and consider the differential form $\delta(F_1)$ on $\mc{H}_3$, for which Proposition \ref{explicit-ES} gives an explicit formula.

Restriction to the upper half-plane $\mc{H}_2 \subset \mc{H}_3$ amounts to setting $x = \overline{x}$, so we obtain
\[
\delta(F_1)|_\Q = (-1)^n w \cdot \sum_{0 \leq \mb{j} \leq \mb{n}} (-1)^{j_2} \binom{\mb{n}}{\mb{j}} X^{\mb{n}-\mb{j}} Y^\mb{j} (G_{n+j_1-j_2} + G_{n+j_1-j_2+2}) y^{-2} dy \wedge dx
\]
as a differential form on $\mc{H}_2$ with values in $V_{n,n}(\C)$.

\subsubsection{Projection onto $1$-dimensional subrepresentation}

While $V_{n,n}(\C)$ is an irreducible representation of $\SL_2(\O_K)$, it is no longer irreducible as a module over $\SL_2(\Z)$. The following decomposition is well-known.

\begin{lem}[Clebsch--Gordan decomposition]
Let $R$ be any $\Z[\frac{1}{n}]$-algebra. There is an isomorphism
\[
V_{n,n}(R) \cong \bigoplus_{i=0}^n V_{2n-2i}(R)
\]
of modules over $\SL_2(\Z)$, given by $\oplus (i!)^{-2} \nabla^i$, where $\nabla$ is the differential operator
\[
\nabla = \frac{\partial^2}{\partial X_2 \partial Y_1} - \frac{\partial^2}{\partial X_1 \partial Y_2}.
\]
\end{lem}

This decomposition induces
\[
H_\cusp^2(Y_\Q, V_{n,n}(\C)) \cong \bigoplus_{i=0}^n H_\cusp^2(Y_\Q, V_{2n-2i}(\C)).
\]
The component that is relevant for the special value $L(1, \ad(f) \otimes \alpha)$ turns out to be $V_0(\C) = \C$, the trivial coefficients. Denoting this projection as
\[
\pi: H_\cusp^2(Y_\Q, V_{n,n}(\C)) \to H_\cusp^2(Y_\Q, \C),
\]
we can easily check that
\[
\pi(\delta(F_1)|_\Q) = \sum_{j=0}^n (G_{2j} + G_{2j+2}) y^{-2} dy \wedge dx.
\]

\subsubsection{A Rankin--Selberg integral}

The special value formula is proved using Rankin--Selberg convolution with a suitable Eisenstein series, which we introduce now.

For any Dirichlet character $\psi: (\Z/N\Z)^\times \to \C^\times$ (not necessarily primitive), define the Eisenstein series
\[
E(z; s, \psi) := L_N(2s, \psi^{-1}) y^s \sum_{\gamma \in U \backslash \Gamma_0(N)} \psi(\gamma) |j(\gamma, z)|^{-2s},
\]
where $U = \left\{ \pm \begin{pmatrix} 1 & m \\ 0 & 1 \end{pmatrix}: m \in \Z \right\}$. We record its properties as follows:
\begin{itemize}
\item By a standard group-theoretic argument, we have
\[
E(z; s, \psi) = \frac{1}{2} y^s \sum_{\substack{(m,n) \in \Z^2 \\ (m,n) \neq (0,0)}} \psi(n) |mNz + n|^{-2s}.
\]
\item $E(z; s, \psi)$ has meromorphic continuation for all $s \in \C$, with Fourier expansion
\[
E \left( -\frac{1}{Nz}; s, \psi \right) = 2^{1-2s} N^{-s} (2\pi y)^{1-s} \frac{\Gamma(2s-1)}{\Gamma(s)^2} L_N(2s-1, \psi) + \text{(entire function of $s$)}.
\]
If $\psi$ is non-trivial, then $E(z; s, \psi)$ is holomorphic at $s = 1$; otherwise, it has a simple pole with residue $2^{-1} \pi N^{-2} \phi(N)$ (note that the extra factor of $N^{-1} \phi(N)$ reflects the absence from $L_N(s, \psi)$ of Euler factors at $N$).
\end{itemize}

Let us denote by $\id_N: (\Z/N\Z)^\times \to \C^\times$ the trivial Dirichlet character mod $N$.

\begin{prop}
\begin{align}
\label{RS-untwisted}
&\int_{\Gamma_0(N) \backslash \mc{H}_2} E(z; s, \id_N) \cdot \pi(\delta(F_1)|_\Q) \nonumber \\
=& \frac{(1 + (-1)^{n+1}) 2^{n+s-1}}{(4 \pi D^{-1/2})^{n+s+1}} \frac{\Gamma(\frac{s}{2})^2 \Gamma(s+n+1)}{\Gamma(s)} L(s+n+1, \As(F))
\end{align}
where $-D$ is the discriminant of $K$, and $L(s, \As(F))$ is the Asai $L$-function of $F$.
\end{prop}

\begin{proof}[Proof (sketch)]
This follows from the Rankin--Selberg method. As an outline:
\begin{enumerate}
\item A direct calculation involving properties of Bessel functions shows that
\[
\int_{U \backslash \mc{H}_2} y^s \cdot \pi(\delta(F_1)|_\Q) = \frac{(1 + (-1)^{n+1}) 2^{n+s-1}}{(4 \pi D^{-1/2})^{n+s+1}} \frac{\Gamma(\frac{s}{2})^2 \Gamma(s+n+1)}{\Gamma(s)} L_{K/\Q}(s+n, F).
\]

\item Then unfolding turns the left-hand side into
\begin{align*}
\int_{U \backslash \mc{H}_2} y^s \cdot \pi(\delta(F_1)|_\Q) &= \int_{\Gamma_0(N) \backslash \mc{H}_2} y^s \sum_{\gamma \in \U \backslash \Gamma_0(N)} \id_N(\gamma) |j(\gamma, z)|^{-2s} \cdot \pi(\delta(F_1)|_\Q) \\
&= L(2s, \id_N)^{-1} \int_{\Gamma_0(N) \backslash \mc{H}_2} E(z; s, \id_N) \cdot \pi(\delta(F_1)|_\Q).
\end{align*}

\item Finally, the definition of the Asai $L$-function gives
\[
L(s+n+1, \As(F)) = L(2s, \id_N) L_{K/\Q}(s+n, F). \qedhere
\]
\end{enumerate}
\end{proof}

Now we are ready to state Hida's integral formula \cite{H99} in our setting.

\begin{thm}[Hida]
Suppose $f \in S_{n+2}(N, \alpha)$ is a normalized eigenform, with base-change Bianchi form $F \in S_{n,n}(U_0(N))$. Then
\begin{equation}
\label{Hida-integral}
L_N(1, \ad(f) \otimes \alpha)
= \frac{(2 \pi D^{-1/2})^{n+2} \phi(N)}{N^2 (n+1)!} \int_{\Gamma_0(N) \backslash \mc{H}_2} \pi(\delta(F_1)|_\Q).
\end{equation}
\end{thm}

\begin{proof}
This follows immediately by taking residues at $s=1$ on both sides of (\ref{RS-untwisted}) and using the factorization from Lemma \ref{factorization}:
\[
L(s+n+1, \As(F)) = L(s, \id_N) L(s, \ad(f) \otimes \alpha)
\]
which is an equality up to finitely many Euler factors at $N$. Removing those Euler factors gives the desired identity.
\end{proof}

For general eigenforms $f \in S_k(N, \psi)$ with Nebentype $\psi$, $L(1, \ad(f) \otimes \alpha)$ can be obtained by twisting $f_K$ by a suitable Hecke character $\varphi$ on $K$. Since such a twist significantly complicates the notation, the construction of the $p$-adic $L$-function in the general case will appear elsewhere.

\subsection{Algebraicity and cohomological interpretation}

In this subsection, we interpret Hida's integral formula (\ref{Hida-integral}) for $L(1, \ad(f) \otimes \alpha)$ in terms of cohomology.

For the purpose of $p$-adic interpolation, it will be easier to work with compactly-supported cohomology $H_c^q$ rather than cuspidal cohomology $H_\cusp^q$, since the former enjoys better functorial properties. To make this shift, recall from \cite{H99} that the natural surjection $H_c^2(Y_K, V_{n,n}(\C)) \twoheadrightarrow H_\cusp^2(Y_K, V_{n,n}(\C))$ admits a canonical section
\[
H_\cusp^2(Y_K, V_{n,n}(\C)) \hookrightarrow H_c^2(Y_K, V_{n,n}(\C)).
\]
Thus every harmonic $2$-form considered in the previous subsection will be implicitly viewed as a cohomologous compactly-supported $2$-form.

\begin{rem}
Elements of $H_c^q$ on a modular variety are often called modular symbols in the literature. When $q=1$, these can be realized as explicit homomorphisms on the divisor group of cusps (see \cite{AS} and \cite{W17}), but no such identification is available for general degree $q$.
\end{rem}

Suppose $E$ is a number field containing all the Hecke eigenvalues of $F$. By multiplicity one \cite{H94-L}, the $F$-isotypic component $H_c^2(Y_K, V_{n,n}(E))$ is a one-dimensional vector space over $E$. To work with integral coefficients, suppose $\O$ is the localization of the ring of integers of $E$ at a prime above $p$. Since $\O$ is a discrete valuation ring, the $F$-isotypic component of $H_c^2(Y_K, V_{n,n}(\O))$ is free of rank one over $\O$. Following \cite{U95}, we make:

\begin{defn}
\label{period}
There exists a period $u_2(F) \in \C^\times$, unique up to multiplication by $\O^\times$, such that
\[
\widehat{\delta}_2(F) := \frac{1}{u_2(F)} \cdot \delta_2(F) \in H_c^2(Y_K, V_{n,n}(\O))
\]
gives an $\O$-basis of $H_c^2(Y_K, V_{n,n}(\O))[F]$.
\end{defn}

Hida's integral expression can be interpreted as a formula for the composition
\[
\xymatrix{
H_c^2(Y_K, V_{n,n}(\C)) \ar[r]^-{|_\Q} & H_c^2(Y_\Q, V_{n,n}(\C)) \ar[r]^-{\pi} & H_c^2(Y_\Q, \C) \ar[r]^-{\sim} & \C.
}
\]
Note that each of these steps is defined rationally over $E$, so we may consider
\[
\xymatrix{
H_c^2(Y_K, V_{n,n}(E)) \ar[r]^-{|_\Q} & H_c^2(Y_\Q, V_{n,n}(E)) \ar[r]^-{\pi} & H_c^2(Y_\Q, E) \ar[r]^-{\sim} & E.
}
\]

\begin{prop}
\label{alg-part}
For every normalized eigenform $f \in S_{n+2}(N, \alpha)$, define the algebraic part of $L(1, \ad(f) \otimes \alpha)$ to be
\[
L^\alg(1, \ad(f) \otimes \alpha) := \frac{N^2 (k-1)!}{(2 \pi D^{-1/2})^k \phi(N)} \cdot \frac{L_N(1, \ad(f) \otimes \alpha)}{u_2(f_K)}.
\]
Then
\[
L^\alg(1, \ad(f) \otimes \alpha) \in E.
\]
Moreover, if either $p > n$ or $f$ is ordinary, then
\[
L^\alg(1, \ad(f) \otimes \alpha) \in \O.
\]
\end{prop}

\begin{proof}
By Hida's integral formula (\ref{Hida-integral}), we see that
\[
L^\alg(1, \ad(f) \otimes \alpha) = \int_{\Gamma_0(N) \backslash \mc{H}_2} \pi(\widehat{\delta}(F_1)|_\Q).
\]
The normalized form $\widehat{\delta}(F_1)$ has coefficients in $\O \subset E$, so the evaluation map above yields an element of $E$.

Note that the restriction map $|_\Q$ and integration $\int_{Y_\Q}$ are defined integrally over $\O$, but the projection map $\pi = (n!)^{-1} \nabla^n$ might contain $n!$ as a denominator. If $p > n$, then $n!$ is invertible in $\O$, so $L^\alg(1, \ad(f) \otimes \alpha) \in \O$.

Finally, if $f$ is ordinary, standard arguments \cite{H88b} show that $n!$ does \emph{not} show up in the denominator, so the same conclusion holds. Alternatively, this will follow from our $p$-adic interpolation, in which the $\Lambda$-adic version of $\pi$ is defined over $\O$.
\end{proof}

\section{\texorpdfstring{$p$}{p}-adic interpolation via measures}

\label{3}

In this section, we construct a $p$-adic interpolation for the evaluation map.

\subsection{Setting and overview}

For the rest of this paper, fix an embedding $\overline{\Q} \hookrightarrow \overline{\Q_p}$, so that algebraic numbers can be viewed as $p$-adic numbers. The valuation ring $\O$ will be replaced by its $p$-adic completion, which is a $p$-adic ring of integers; again this will not cause any confusion.

Recall that Hida's integral formula can be interpreted cohomomologically as an evaluation map
\[
\xymatrix{
\mathbb{L}_k: H_c^2(Y_K, V_{k,k}(\O)) \ar[r]^-{|_\Q} & H_c^2(Y_\Q, V_{k,k}(\O)) \ar[r]^-{\pi} & H_c^2(Y_\Q, \O) \ar[r]^-{\sim} & \O
}
\]
such that for $f \in S_{k+2}(N, \alpha)$,
\[
\mathbb{L}_k(\widehat{\delta}_2(f_K)) = L^\alg(1, \ad(f) \otimes \alpha).
\]

\begin{rem}
Strictly speaking, the projection $\pi$ should map into $H_c^2(Y_\Q, \O[\frac{1}{k!}])$ rather than $H_c^2(Y_\Q, \O)$. This simplification in notation will be harmless, since our ultimate interest is in constructing $p$-adic $L$-functions on ordinary families. The reader might prefer to replace all cohomology groups $H_c^2$ with their ordinary parts $H_{c,\ord}^2$.
\end{rem}

\begin{hyp}
$p$ is an odd prime which is split in $K$.
\end{hyp}

Fix an identification of $\O_{K,p} := \O_K \otimes_\Z \Z_p$ with $\Z_p \times \Z_p$, which is induced by the two embeddings $K \hookrightarrow \overline{\Q_p}$. This hypothesis only serves to simplify notation; the case of inert $p$ can be dealt with easily.

From now on, the level $N$ will always be divisible by $p$. As $N$ will not change in what follows, we continue to write the Bianchi threefold and modular curve as
\begin{align*}
Y_K &= \Gamma_0^K(N) \backslash \mc{H}_3, \\
Y_\Q &= \Gamma_0^\Q(N) \backslash \mc{H}_2.
\end{align*}

We will consider local systems on $Y_K$ arising from $\O$-modules $M$ with an action of $\Sigma_0(p) \times \Sigma_0(p)$, where $\Sigma_0(p)$ is the monoid
\[
\Sigma_0(p) = \left\{ \begin{pmatrix} a & b \\ c & d \end{pmatrix} \in M_2(\Z_p) \cap \GL_2(\Q_p): c \in p \Z_p, d \in \Z_p^\times \right\}.
\]
In particular, the congruence subgroup $\Gamma_0^K(N)$ acts on such an $M$ via the embedding
\[
\Gamma_0^K(N) \hookrightarrow \Sigma_0(p) \times \Sigma_0(p)
\]
induced by $\O_K \hookrightarrow \O_{K,p} = \Z_p \times \Z_p$. Thus it makes sense to consider the cohomology groups $H_*^q(Y_K, M)$, which are equipped with an action of Hecke operators; the action of the $U_p$-operator will be recalled explicitly in the next section.

Under the diagonal embedding $\Sigma_0(p) \subset \Sigma_0(p)^2$, the restricted action on $M$ gives $H_*^q(Y_\Q, M)$.

Let $A$ be any complete $\Z_p$-algebra. Recall the weight space $\mc{W}$, whose $A$-valued points are $\Hom_{\mathrm{cont}}(\Z_p^\times, A^\times)$.

We first describe the space of $A$-valued measures $\mc{D}_n(A)$ on $\O_{K,p}$, which is equipped with a weight $k$ action and specialization maps $\rho_k: \mc{D}_k(A) \to V_{k,k}(A)$ for every $n$.

For the Iwasawa algebra $\Lambda = \O[[\Z_p^\times]]$ (with tautological character $\theta: \Z_p^\times \to \Lambda^\times$), we consider $\mc{D}(\Lambda)$ (equipped with a canonical weight $\theta$ action), which should be viewed as a space parametrizing $p$-adic families of measures on $\O_{K,p}$. We will see that every continuous character $\kappa: \Z_p^\times \to \O^\times$ induces a specialization map $\sp_\kappa: \Lambda \to \O$ and hence $\sp_\kappa: \mc{D}(\Lambda) \to \mc{D}_\kappa(\O)$.

Combining these, we obtain the specialization maps
\[
\mc{D}(\Lambda) \to V_{k,k}(\O)
\]
and form the diagram:
\begin{equation}
\label{guiding}
\xymatrix{
H_c^2(Y_K, \mc{D}(\Lambda)) \ar[d] \ar@{.>}[r] & H_c^2(Y_\Q, \mc{D}(\Lambda)) \ar[d] \ar@{.>}[r] & H_c^2(Y_\Q, \Lambda) \ar[d] \ar@{.>}[r]^-\sim & \Lambda \ar[d] \\
H_c^2(Y_K, V_{k,k}(\O)) \ar[r]^-{|_\Q} & H_c^2(Y_\Q, V_{k,k}(\O)) \ar[r]^-{\pi} & H_c^2(Y_\Q, \O) \ar[r]^-{\sim} & \O
}
\end{equation}
where the bottom row is Hida's evaluation map. Our goal is to fill in the dotted arrows and construct a ``big'' evaluation map
\[
\mathbb{L}_\Lambda: H_c^2(Y_K, \mc{D}(\Lambda)) \to \Lambda
\]
which specializes to
\[
\mathbb{L}_k: H_c^2(Y_K, V_{k,k}(\O)) \to \O
\]
for every $k$.

Unfortunately this will not literally work out; as we shall see later, the projection map $\mc{D}(\Lambda) \to \Lambda$ can only be defined on measures supported on a certain open subset $\O_{K,p}' \subseteq \O_{K,p}$. The failure of commutativity of this diagram will result in an Euler factor at $p$.

\subsection{Generalities on \texorpdfstring{$p$}{p}-adic measures}

We begin by considering matters in great generality. Throughout this section, $X$ will be a compact totally disconnected topological space, and $A$ will denote a complete topological $\Z_p$-algebra.

\begin{defn}
Let $\mc{C}(X; A)$ (resp.\ $\mc{C}^\infty(X; A)$) be the space of continuous (resp.\ locally constant) $A$-valued functions on $X$, equipped with the compact-open topology. The space of $A$-valued measures is defined to be
\[
\mc{D}(X; A) = \Hom_{A}(\mc{C}^\infty(X; A), A).
\]
These are abbreviated as $\mc{C}(A)$, $\mc{C}^\infty(A)$ and $\mc{D}(A)$ whenever there is no ambiguity about the base space $X$.
\end{defn}

By the density of $\mc{C}^\infty(A) \subset \mc{C}(A)$, it is easy to see that every $\mu \in \mc{D}(A)$ has a unique extension to a continuous $A$-linear homomorphism $\mc{C}(A) \to A$. Thus $\mc{D}(A)$ is isomorphic to the continuous $A$-linear dual of $\mc{C}(A)$.

\begin{rem}
If $A$ is a Banach $\Q_p$-algebra, the compact-open topology on $\mc{C}(A)$ coincides with the metric topology induced by the sup-norm.
\end{rem}

It is important to address the behavior of our construction under base change $\phi: A \to A'$, which is not available in the literature to the best of our knowledge. Since $X$ is compact, the natural map
\[
\mc{C}^\infty(A) \otimes_A A' \overset{\sim}{\to} \mc{C}^\infty(A')
\]
is an isomorphism. However, the map $\mc{C}(A) \otimes_A A' \to \mc{C}(A')$ is in general not an isomorphism.

For the space of measures, the map $\mc{D}(A) \to \mc{D}(A')$ can be defined as follows. For any $A$-linear map $\mu: \mc{C}^\infty(A) \to A$, extension of scalars to $A'$ gives
\[
\mu \otimes 1_{A'}: \mc{C}^\infty(A) \otimes_A A' \to A \otimes_A A',
\]
which is naturally an $A'$-linear map $\mc{C}^\infty(A') \to A'$. For the construction of specialization maps, we have to be more precise about the natural identifications involved, so we make:

\begin{defn}
\label{measure-base-change}
For an algebra homomorphism $\phi: A \to A'$, the base change map $\phi: \mc{D}(A) \to \mc{D}(A')$ is defined by sending $\mu \in \mc{D}(A)$ to the composite map
\[
\xymatrix{
\mc{C}^\infty(A') \ar[r]^-{i}_-{\sim} & \mc{C}^\infty(A) \otimes_A A' \ar[rr]^-{\mu \otimes 1_{A'}} && A \otimes_A A' \ar[rr]^-{\phi \otimes 1_{A'}}_-{\sim} && A',
}
\]
where $i$ is the inverse of the natural isomorphism $\mc{C}^\infty(A) \otimes_A A' \to \mc{C}^\infty(A')$.
\end{defn}

Unfortunately, $\mc{D}(A)$ does not behave well under base change. More precisely, the natural map $\mc{D}(A) \otimes_A A' \to \mc{D}(A')$ is not surjective in general, but the approach above suffices for our purpose. Alternatively, one may circumvent these issues by systematically introducing completed tensor products.

We end this general discussion by stating that the base change map on measures behaves well with respect to evaluation at functions, which is straightforward from the definitions.

\begin{prop}
\label{measure-evaluation}
Let $\phi: A \to A'$ be an algebra homomorphism, inducing $\phi: \mc{D}(A) \to \mc{D}(A')$. Then for every $f \in \mc{C}(A)$, the diagram
\[
\xymatrix{
\mc{D}(A) \ar[r] \ar[d]_{\text{evaluation at $f$}} & \mc{D}(A') \ar[d]^{\text{evaluation at $\phi \circ f$}} \\
A \ar[r]^\phi & A'
}
\]
commutes.
\end{prop}

One might think of this as saying ``integration is natural''.

\begin{rem}
Measures, as considered in the topological setting, are sufficient for studying ordinary families. To extend our construction to finite-slope families on the eigencurve, we have to replace $\mc{D}(A)$ by the larger space of \emph{locally analytic} distributions.
\end{rem}

\subsection{Spaces of measures and polynomials}

We will define certains spaces of measures and polynomials which are equipped with actions of $\Sigma_0(p) \times \Sigma_0(p)$, where $\Sigma_0(p)$ is the monoid
\[
\Sigma_0(p) = \left\{ \begin{pmatrix} a & b \\ c & d \end{pmatrix} \in M_2(\Z_p) \cap \GL_2(\Q_p): c \in p \Z_p, d \in \Z_p^\times \right\}.
\]
These spaces will also be equipped with a restricted action of $\Sigma_0(p)$, via the diagonal embedding $\Sigma_0(p) \subset \Sigma_0(p) \times \Sigma_0(p)$.

\subsubsection{Polynomials}

Now we consider a $p$-adic version of the space of homogeneous polynomials.

For every $k \in \Z_{\geq 0}$, let
\[
V_{k,k}(A) = \Sym^k A^2 \otimes \Sym^k A^2
\]
be the space of polynomials in $(X_1, Y_1, X_2, Y_2)$ with coefficients in $A$ homogeneous of degree $k$ for each pair $(X_i, Y_i)$, which is equipped with a left action of $\Sigma_0(p) \times \Sigma_0(p)$ by
\[
(\gamma_1, \gamma_2) \cdot P(X,Y) = \bigotimes_{i=1}^2 P(d_iX_i-b_iY_i, -c_iX_i+a_iY_i).
\]
This clearly agrees with the space $V_{k,k}(A)$ defined in Section \ref{section-cohomology}, with the restricted action of $\Sigma_0(p)^2$ as a subgroup of $\GL_2(A)^2$.

\subsubsection{Measures}

Recall $\O_{K,p} = \O_K \otimes_\Z \Z_p \simeq \Z_p \times \Z_p$. We set
\begin{align*}
\mc{C}(A) &:= \mc{C}(\O_{K,p}; A), \\
\mc{D}(A) &:= \mc{D}(\O_{K,p}; A).
\end{align*}
For every weight $\kappa \in \mc{W}(A)$, i.e., a continuous character $\Z_p^\times \to A^\times$, there is a natural right action on $\mc{C}(A)$ by $\Sigma_0(p) \times \Sigma_0(p)$
via the formula
\[
(f |_\kappa (\gamma_1, \gamma_2))(z_1, z_2) = \kappa(c_1z_1+d_1) \kappa(c_2z_2+d_2) f \left( \frac{a_1z_1+b_1}{c_1z_1+d_1}, \frac{a_2z_2+b_2}{c_2z_2+d_2} \right).
\]

Denote by $\mc{C}_\kappa(A)$ the space $\mc{C}(A)$ equipped with this weight $\kappa$ action, and $\mc{D}_\kappa(A)$ the corresponding left action on $\mc{D}(A)$ by duality:
\[
(\gamma \cdot_\kappa \mu)(f) = \mu(f |_\kappa \gamma).
\]
If $\kappa(t) = t^k$ is an integral weight, we often write $k$ in place of $\kappa$ when there is no ambiguity, such as $\mc{C}_k(A)$ and $\mc{D}_k(A)$.

\subsection{Specialization maps}

\label{section-specialization}

The goal is to define specialization maps out of $\mc{D}(\Lambda)$ to all the polynomial spaces.

\begin{rem}
In this subsection only, formulas are sometimes written in one variable for notational simplicity. For example, the actions on $\mc{C}_\kappa(A)$ and $V_{k,k}(A)$ may respectively be denoted
\[
(f |_\kappa \gamma)(z) = \kappa(cz+d) f \left( \frac{az+b}{cz+d} \right)
\]
and
\[
(\gamma \cdot P)(X,Y) = P(dX-bY, -cX+aY)
\]
for $\gamma \in \Sigma_0(p)^2$.
\end{rem}

\subsubsection{From measures to polynomials}

For each $k \in \Z_{\geq 0}$, the weight $k$ specialization map
\[
\rho_k: \mc{D}_k(A) \to V_{k,k}(A)
\]
is defined by
\begin{align*}
\mu &\mapsto \int (X_1 - z_1 Y_1)^k (X_2 - z_2 Y_2)^k \,d\mu(z_1, z_2) \\
&= \int \left( \sum_{j=0}^k (-1)^j \binom{k}{j} z_1^j X_1^{k-j} Y_1^j \right) \left( \sum_{\ell=0}^k (-1)^\ell \binom{k}{\ell} z_2^\ell X_2^{k-\ell} Y_2^\ell \right) \,d\mu(z_1, z_2).
\end{align*}

\begin{prop}
The map $\rho_k: \mc{D}_k(A) \to V_{k,k}(A)$ is $\Sigma_0(p)^2$-equivariant.
\end{prop}

\begin{proof}
This is straightforward, but we provide the full calculation (in one-variable notation) since different conventions are used in the literature.

By definition,
\begin{align*}
\rho_k(\gamma \cdot_k \mu) &= \sum_{j=0}^k (-1)^j \binom{k}{j} (\gamma \cdot_k \mu)(z^j) X^{k-j} Y^j \\
&= \sum_{j=0}^k (-1)^j \binom{k}{j} \mu \left( (cz+d)^k \left(\frac{az+b}{cz+d}\right)^j \right) X^{k-j} Y^j \\
&= \sum_{j=0}^k (-1)^j \binom{k}{j} \mu \left( (az+b)^j (cz+d)^{k-j} \right) X^{k-j} Y^j.
\end{align*}
Expanding $(az+b)^j (cz+d)^{k-j}$, we see that the coefficient of $\mu(z^\ell)$ in the entire expression equals (where $\binom{q}{r} = 0$ whenever $q < r$ or $r < 0$)
\begin{align*}
& \sum_{j=0}^k (-1)^j \binom{k}{j} \left( \sum_{i=0}^\ell \binom{j}{i} a^i b^{j-i} \cdot \binom{k-j}{\ell-i} c^{\ell-i} d^{k-j-\ell+i} \right) X^{k-j} Y^j \\
=& \sum_{j=0}^k \sum_{i=0}^\ell (-1)^j \binom{k}{\ell} \binom{\ell}{i} \binom{k-\ell}{j-i} a^i c^{\ell-i} \cdot b^{j-i} d^{k-j-\ell+i} X^{k-j} Y^j \\
=& \binom{k}{\ell} \sum_{i=0}^\ell \sum_{j=0}^k \binom{\ell}{i} (-aY)^i (cX)^{\ell-i} \cdot \binom{k-\ell}{j-i} (-bY)^{j-i} (dX)^{k-j-\ell+i} \\
=& \binom{k}{\ell} \sum_{i=0}^\ell \binom{\ell}{i} (-aY)^i (cX)^{\ell-i} \sum_{j'=-i}^{k-i} \binom{k-\ell}{j'} (-bY)^{j'} (dX)^{k-\ell-j'} \quad \text{[setting $j'=j-i$]} \\
=& \binom{k}{\ell} \sum_{i=0}^\ell \binom{\ell}{i} (-aY)^i (cX)^{\ell-i} \cdot \sum_{j'=0}^{k-\ell} \binom{k-\ell}{j'} (-bY)^{j'} (dX)^{k-\ell-j'}\\
=& \binom{k}{\ell} (cX-aY)^\ell (dX-bY)^{k-\ell} \\
=& (-1)^\ell \binom{k}{\ell} (dX-bY)^{k-\ell} (-cX+aY)^\ell.
\end{align*}
Therefore we obtain
\begin{align*}
\rho_k(\gamma \cdot_k \mu) &= \sum_{\ell=0}^k (-1)^\ell \binom{k}{\ell} \mu(z^\ell) (dX-bY)^{k-\ell} (-cX+aY)^\ell \\
&= \gamma \cdot \left( \sum_{\ell=0}^k (-1)^\ell \binom{k}{\ell} \mu(z^\ell) X^{k-\ell} Y^\ell \right) \\
&= \gamma \cdot \rho_k(\mu)
\end{align*}
as desired.
\end{proof}

\begin{cor}
The induced maps on cohomology
\[
\rho_k: H_c^2(Y_K, \mc{D}_k(A)) \to H_c^2(Y_K, V_{k,k}(A))
\]
and
\[
\rho_k: H_c^2(Y_\Q, \mc{D}_k(A)) \to H_c^2(Y_\Q, V_{k,k}(A))
\]
are equivariant under the Hecke action.
\end{cor}

\subsubsection{From families to specific weights}

Recall the Iwasawa algebra $\Lambda = \O[[\Z_p^\times]]$ and the space of families of measures $\mc{D}(\Lambda)$. Since $\mc{C}(\Lambda)$ and $\mc{D}(\Lambda)$ are equipped with a canonical action induced by $\theta: \Z_p^\times \to \Lambda$, it is unnecessary to use the more complicated notation $\mc{C}_\theta(\Lambda)$ and $\mc{D}_\theta(\Lambda)$.

Given any weight $\kappa: \Z_p^\times \to \O^\times$, the universal property of Iwasawa algebra gives a unique $\O$-algebra homomorphism $\sp_\kappa: \Lambda \to \O$ which fits into the commutative diagram
\[
\xymatrix{
\Z_p^\times \ar[r]^\theta \ar[rd]_\kappa & \Lambda \ar[d]^{\sp_\kappa} \\
& \O.
}
\]
This is the weight $\kappa$ specialization map on $\Lambda$, which induces specialization maps $\mc{C}(\Lambda) \to \mc{C}(\O)$ and $\mc{D}(\Lambda) \to \mc{D}(\O)$. To check that this is compatible with the $\Sigma_0(p)^2$-action, we begin with:

\begin{lem}
The map $\sp_\kappa: \mc{C}(\Lambda) \to \mc{C}_\kappa(\O)$ is $\Sigma_0(p)^2$-equivariant.
\end{lem}

\begin{proof}
For $\gamma = \begin{pmatrix} a & b \\ c & d \end{pmatrix} \in \Sigma_0(p)^2$ and $f \in \mc{C}(\Lambda)$, we have
\begin{align*}
\sp_\kappa(f |_\theta \gamma)(z) &= \sp_\kappa \circ (f |_\theta \gamma) (z) \\
&= \sp_\kappa \left( \theta(cz+d) f \left( \frac{az+b}{cz+d} \right) \right) \\
&= \kappa(cz+d) (\sp_\kappa \circ f)\left( \frac{az+b}{cz+d} \right) \\
&= (\sp_\kappa (f) |_\kappa \gamma)(z). \qedhere
\end{align*}
\end{proof}

As a consequence, the natural isomorphism $\sp_\kappa \otimes 1_{\O}: \mc{C}^\infty(\Lambda) \otimes_{\Lambda, \sp_\kappa} \O \overset{\sim}{\to} \mc{C}^\infty(\O)$, as well as its inverse $i_\kappa = (\sp_\kappa \otimes 1_{\O})^{-1}$, is also $\Sigma_0(p)^2$-equivariant.

Recall from Definition \ref{measure-base-change} that the map $\mc{D}(\Lambda) \to \mc{D}(\O)$ sends every measure $\mu \in \mc{D}(\Lambda) = \Hom_\Lambda(\mc{C}^\infty(\Lambda), \Lambda)$ to the composition
\[
\xymatrix{
\mc{C}^\infty(\O) \ar[r]^-{i_\kappa}_-{\sim} & \mc{C}^\infty(\Lambda) \otimes_{\Lambda, \sp_\kappa} \O \ar[rr]^-{\mu \otimes 1_{\O}} && \Lambda \otimes_{\Lambda, \sp_\kappa} \O \cong \O.
}
\]

\begin{prop}
The specialization map $\sp_\kappa: \mc{D}(\Lambda) \to \mc{D}_\kappa(\O)$ is $\Sigma_0(p)^2$-equivariant.
\end{prop}

\begin{proof}
Let us check that $\sp_\kappa(\gamma \cdot_\theta \mu) = \gamma \cdot_\kappa \sp_\kappa(\mu)$ for $\gamma \in \Sigma_0(p)^2$ and $\mu \in \Hom_\Lambda(\mc{C}^\infty(\Lambda), \Lambda)$.

For $f \in \mc{C}^\infty(\O)$, we have
\begin{align}
\sp_\kappa(\gamma \cdot_\theta \mu)(f) &= ((\gamma \cdot_\theta \mu) \otimes 1_{\O}) \circ i_\kappa(f) \nonumber \\
\label{line2} &= (\mu \otimes 1_{\O}) \left( i_\kappa(f)|_\theta \gamma \right) \\
\label{line3} &= (\mu \otimes 1_{\O}) \left( i_\kappa(f|_\kappa \gamma) \right) \\
&= \sp_\kappa(\mu)(f |_\kappa \gamma) \nonumber \\
\label{line5} &= (\gamma \cdot_\kappa \sp_\kappa(\mu))(f),
\end{align}
where (\ref{line2}) and (\ref{line5}) follow from the duality between (left) action on measures and (right) action on functions, and (\ref{line3}) follows from the equivariance of $i_\kappa$.
\end{proof}

\begin{defn}
The specialization at weight $k$ on $\mc{D}(\Lambda)$, denoted $\rho_k$, is the composition
\[
\xymatrix{
\mc{D}(\Lambda) \ar[r]^-{\sp_k} & \mc{D}_k(\O) \ar[r]^-{\rho_k} & V_{k,k}(\O),
}
\]
which is $\Sigma_0(p)^2$-equivariant.
\end{defn}

\begin{cor}
The specialization map $\rho_k: \mc{D}(\Lambda) \to V_{k,k}(\O)$ induces maps on cohomology
\[
\rho_k: H_c^2(Y_K, \mc{D}(\Lambda)) \to H_c^2(Y_K, V_{k,k}(\O))
\]
and
\[
\rho_k: H_c^2(Y_\Q, \mc{D}(\Lambda)) \to H_c^2(Y_\Q, V_{k,k}(\O)),
\]
which are equivariant under the Hecke action.
\end{cor}

\subsection{Evaluation map on families}

\label{section-evaluation}

In this subsection, we construct the $\Lambda$-adic maps in the top row of (\ref{guiding}). The two maps $|_\Q$ and $\int_{Y_\Q}$ can be interpolated in a straightforward manner, by generalities about cohomology.

\subsubsection{Restriction to $\Q$}

\begin{lem}
\label{big-restriction}
The diagram
\[
\xymatrix{
H_c^2(Y_K, \mc{D}(\Lambda)) \ar[d]_{\rho_k} \ar[r]^{|_\Q} & H_c^2(Y_\Q, \mc{D}(\Lambda)) \ar[d]^{\rho_k} \\
H_c^2(Y_K, V_{k,k}(\O)) \ar[r]^{|_\Q} & H_c^2(Y_\Q, V_{k,k}(\O))
}
\]
commutes.
\end{lem}

\begin{proof}
This follows from the $\Sigma_0(p)^2$-equivariance of $\rho_k$ and functoriality for the diagonal embedding $\Sigma_0(p) \subset \Sigma_0(p)^2$.
\end{proof}

\subsubsection{Integration on modular curve}

Over $\C$, integration over the modular curve $Y_\Q$ gives an isomorphism $H_c^2(Y_\Q, \C) \overset{\sim}{\to} \C$. For a trivial coefficient ring $A$ in general, there is still a canonical isomorphism
\begin{align*}
H_c^2(Y_\Q, A) &\overset{\sim}{\to} A \\
\phi &\mapsto \phi \cap [Y_\Q]
\end{align*}
given by cap product $\cap: H_c^2(Y_\Q, A) \otimes_\Z H_2^\mathrm{BM}(Y_\Q, \Z) \to H_0(Y_\Q, A) \cong A$; here $[Y_\Q] \in H_2^\mathrm{BM}(Y_\Q, \Z)$ is a fundamental class of Borel--Moore homology, whose choice will be fixed throughout the paper.

\begin{lem}
\label{big-integration}
The diagram
\[
\xymatrix{
H_c^2(Y_\Q, \Lambda) \ar[d]_{\sp_k} \ar[r]^-{\sim} & \Lambda \ar[d]^{\sp_k} \\
H_c^2(Y_\Q, \O) \ar[r]^-{\sim} & \O
}
\]
commutes.
\end{lem}

\begin{proof}
The canonical isomorphism is natural in $A$.
\end{proof}

\subsubsection{Projection onto trivial coefficients}

The goal here is to motivate how one might define a map $\Pi: \mc{D}(\Lambda) \to \Lambda$ which interpolates the projection maps $(k!)^{-2} \nabla^k: V_{k,k}(\O) \to \O$ for all $n$.

Consider a hypothetical diagram
\[
\xymatrix{
\mc{D}(\Lambda) \ar@{.>}[rr]^-{\Pi} \ar[d]_{\sp_k} && \Lambda \ar[dd]^{\sp_k} \\
\mc{D}_k(\O) \ar[rrd] \ar[d]_{\rho_k} \\
V_{k,k}(\O) \ar[rr]_-{(k!)^{-2} \nabla^k} && \O
}
\]
We check by a direct computation that the diagonal map has a particularly nice description.

\begin{lem}
The composition
\[
\xymatrix{
\mc{D}(A) \ar[r]^-{\rho_k} & V_{k,k}(A) \ar[rr]^-{(k!)^{-2} \nabla^k} && A
}
\]
is given by evaluation at $(z_1-z_2)^k$, i.e., $\mu \mapsto \mu[(z_1-z_2)^k]$.
\end{lem}

\begin{proof}
Recall that $\rho_k$ is evaluation at the polynomial
\begin{align*}
(X_1 - z_1 Y_1)^k (X_2 - z_2 Y_2)^k &= \left( \sum_{j=0}^k (-1)^j \binom{k}{j} X_1^{k-j} Y_1^j z_1^j \right) \left( \sum_{\ell=0}^k (-1)^\ell \binom{k}{\ell} X_2^{k-\ell} Y_2^\ell z_2^\ell \right) \\
&= \sum_{j=0}^k \sum_{\ell=0}^k (-1)^{j+\ell} \binom{k}{j} \binom{k}{\ell} X_1^{k-j} Y_1^j X_2^{k-\ell} Y_2^\ell z_1^j z_2^\ell.
\end{align*}
Since the differential operator $\nabla = \frac{\partial^2}{\partial X_2 \partial Y_1} - \frac{\partial^2}{\partial X_1 \partial Y_2}$ satisfies
\[
(k!)^{-2} \nabla^k X_1^{k-a} Y_1^a X_2^{k-b} Y_2^b =
\begin{cases}
(-1)^a \binom{k}{a}^{-1} &\text{if } k = a+b, \\
0 &\text{otherwise},
\end{cases}
\]
we see that
\begin{align*}
(k!)^{-2} \nabla^k [(X_1 - z_1 Y_1)^k (X_2 - z_2 Y_2)^k] &= \sum_{j=0}^k (-1)^k \binom{k}{j} \binom{k}{k-j} \cdot (-1)^j \binom{k}{j}^{-1} z_1^j z_2^{k-j} \\
&= \sum_{j=0}^k (-1)^{k-j} \binom{k}{j} z_1^j z_2^{k-j} \\
&= (z_1 - z_2)^k
\end{align*}
as desired.
\end{proof}

\begin{rem}
More conceptually, it is easy to check that $(z_1-z_2)^k$ is an invariant function under the Iwahori subgroup
\[
\Iw(p) := \Sigma_0(p) \cap \SL_2(\Z_p) = \left\{ \begin{pmatrix} a & b \\ c & d \end{pmatrix} \in \SL_2(\Z_p): p \mid c \right\},
\]
so evaluating a measure on $(z_1-z_2)^k$ yields the trivial coefficients; this point of view will be exploited in the next section.
\end{rem}

In view of Proposition \ref{measure-evaluation}, we want to seek a function $\O_{K,p} \to \Lambda$ which specializes to $(z_1-z_2)^k$ for all $k$. Unfortunately this is not possible; the function $z \mapsto z^k$ on $\Z_p$ cannot be $p$-adically interpolated in $k$, but on $\Z_p^\times$ it is interpolated by the tautological character $\theta: \Z_p^\times \to \Lambda^\times$. This motivates the following definition, which is the best possible interpolation of $(z_1-z_2)^k$ on $\O_{K,p}$.

\begin{defn}
The map $\Pi: \mc{D}(\Lambda) \to \Lambda$ is defined as the following evaluation:
\[
\Pi(\mu) = \int_{\O_{K,p}} \theta(z_1-z_2) \,d\mu,
\]
where $\theta: \Z_p^\times \to \Lambda^\times$ is the tautological character, extended to $\Z_p$ by zero.
\end{defn}

Note that $\theta(z_1-z_2)$ is supported on the open subspace
\[
\O_{K,p}' := \{(z_1, z_2) \in \O_{K,p}: z_1 - z_2 \in \Z_p^\times\},
\]
so $\Pi: \mc{D}(\Lambda) \to \Lambda$ factors through those measures in $\mc{D}(\Lambda)$ which are supported on $\O_{K,p}'$. As a consequence, the diagram
\[
\xymatrix{
\mc{D}(\Lambda) \ar[r]^-\Pi \ar[d]_{\sp_k} & \Lambda \ar[d]^{\sp_k} \\
\mc{D}_k(\O) \ar[r] & \O
}
\]
\emph{fails} to commute: for $\mu \in \mc{D}(\Lambda)$, going right then down gives
\[
\int_{\O_{K,p}'} (z_1-z_2)^k \,d\mu,
\]
whereas going down then right gives
\[
\int_{\O_{K,p}} (z_1-z_2)^k \,d\mu.
\]

On the level of cohomology, this means the $\Lambda$-adic map $H_c^2(Y_\Q, \mc{D}(\Lambda)) \to H_c^2(Y_\Q, \Lambda)$ does not specialize exactly to $H_c^2(Y_\Q, V_{k,k}(\O)) \to \O$. It turns out that for a $U_p$-eigenclass, their discrepancy can be measured by a suitable Euler factor; this is the subject of the next section.

\begin{rem}
Integrating $(z_1-z_2)^k$ on $\O_{K,p}'$ is the same as integrating $1_{\O_{K,p}'} (z_1-z_2)^k$ on all of $\O_{K,p}$, where $1_S$ denotes the characteristic function of $S$. Similarly as before, $1_{\O_{K,p}'} (z_1-z_2)^k$ is invariant under the Iwahori subgroup $\Iw(p) \subset \SL_2(\Z_p)$.
\end{rem}

\section{Interpolation formula on Hecke eigenclass}

\label{4}

To recap, we have constructed all the maps in the diagram
\[
\xymatrix{
H_c^2(Y_K, \mc{D}(\Lambda)) \ar[r]^-{\mathrm{res}} \ar[d] & H_c^2(Y_\Q, \mc{D}(\Lambda)) \ar[r] \ar[d] & H_c^2(Y_\Q, \Lambda) \ar[r]^-\sim \ar[dd] & \Lambda \ar[dd] \\
H_c^2(Y_K, \mc{D}_k(\O)) \ar[r]^-{\mathrm{res}} \ar[d] & H_c^2(Y_\Q, \mc{D}_k(\O)) \ar[rd] \ar[d] \ar@{}[ru]|{(*)} \\
H_c^2(Y_K, V_k(\O)) \ar[r]^-{\mathrm{res}} & H_c^2(Y_\Q, V_k(\O)) \ar[r] & H_c^2(Y_\Q, \O) \ar[r]^-\sim & \O
}
\]
but the trapezoid $(*)$ does \emph{not} commute; it is induced by
\[
\xymatrix{
\mc{D}(\Lambda) \ar[rrrrr]^-{\textstyle \mu \mapsto \int_{\O_{K,p}} \theta(z_1-z_2) \,d\mu} \ar[d] &&&&& \Lambda \ar[d] \\
\mc{D}_k(\O) \ar[rrrrr]^-{\textstyle \mu \mapsto \int_{\O_{K,p}} (z_1-z_2)^k \,d\mu} &&&&& \O
}
\]
Our next goal is to measure this discrepancy. The calculation turns out to be cleanest if we work at the level of $\mc{D}_k(\O)$, so we extract the middle row of the diagram above. To simplify notation, we denote
\begin{align*}
X &:= \O_{K,p} \simeq \Z_p \times \Z_p, \\
X' &:= \O_{k,p}' = \{(z_1,z_2) \in X: z_1-z_2 \in \Z_p^\times\}.
\end{align*}

\begin{defn}
Define the evaluation map $\Ev_X$ (resp. $\Ev_{X'}$) to be the composition
\[
\xymatrix{
H_c^2(Y_K, \mc{D}_k(\O)) \ar[r]^-{\mathrm{res}} & H_c^2(Y_\Q, \mc{D}_k(\O)) \ar[r] & H_c^2(Y_\Q, \O) \ar[r]^-{\sim} & \O,
}
\]
where the middle map is induced by the evaluation map sending $\mu \in \mc{D}_k(\O)$ to
\[
\int_X (z_1-z_2)^k \,d\mu \quad \left( \text{resp. $\int_{X'} (z_1-z_2)^k \,d\mu$} \right).
\]
\end{defn}

Recall that $\Ev_X$ encodes an actual $L$-value, but the $p$-adic $L$-function will only specialize to $\Ev_{X'}$. We shall see that on a $U_p$-eigenclass of $H_c^2(Y_K, \mc{D}_k(\O))$, they differ by an Euler factor at $p$; the situation is similar to the difference between the standard and improved $p$-adic $L$-functions in \cite{GS93}.

\subsection{Cohomological interpretation}

Our first step is to interpret each map in terms of cohomological operations:
\[
\xymatrix{
\Ev_X: H_c^2(Y_K, \mc{D}_k(\O)) \ar[r]^-{\mathrm{res}}_-{\iota^*} & H_c^2(Y_\Q, \mc{D}_k(\O)) \ar[rr]^-{\int_X}_-{\cup [(z_1-z_2)^k]} && H_c^2(Y_\Q, \O) \ar[r]^-{\sim}_-{\cap [Y_\Q]} & \O.
}
\]
\begin{enumerate}
\item The first map is restriction along the inclusion $\iota: Y_\Q \hookrightarrow Y_K$, and can be thought of as the pullback $\iota^*$.
\item The second map is induced by the evaluation of measures:
\begin{align*}
\mc{D}_k(\O) &\to \O \\
\mu &\mapsto \mu[(z_1-z_2)^k] = \int_X (z_1-z_2)^k \,d\mu.
\end{align*}
Viewing the function $(z_1-z_2)^k$ as a class $[(z_1-z_2)^k] \in H^0(Y_\Q, \mc{C}_k(\O))$, we can interpret the map as given by the cup product pairing
\[
\cup: H_c^2(Y_\Q, \mc{D}_k(\O)) \otimes_\O H^0(Y_\Q, \mc{C}_k(\O)) \to H_c^2(Y_\Q, \mc{D}_k(\O) \otimes_\O \mc{C}_k(\O)) \to H_c^2(Y_\Q, \O)
\]
with $[(z_1-z_2)^k] \in H^0(Y_\Q, \mc{C}_k(\O))$.

\item The last map is given by cap product
\[
\cap: H_c^2(Y_\Q, \O) \otimes_\Z H_2^\mathrm{BM}(Y_\Q, \Z) \to H_0(Y_\Q, \O) \cong \O
\]
with the fundamental class $[Y_\Q] \in H_2^\mathrm{BM}(Y_\Q, \Z)$.
\end{enumerate}

A similar description is available for $\Ev_{X'}$, by replacing $\int_X$ with $\int_{X'}$ in the evaluation map $\mc{D}_k(\O) \to \O$. Then the middle map $H_c^2(Y_\Q, \mc{D}_k(\O)) \to H_c^2(Y_\Q, \O)$ can be thought of as cup product with the class
\[
[1_{X'}(z_1-z_2)^k] \in H^0(Y_\Q, \mc{C}_k(\O)),
\]
where $1_{X'} = 1_{X'}(z_1,z_2)$ is the characteristic function of $X' = \{(z_1, z_2) \in X: z_1-z_2 \in \Z_p^\times\}$.

Under this formulation, the goal is to find, for a $U_p$-eigenclass $\phi \in H_c^2(Y_K, \mc{D}_k(\O))$, the difference between
\[
\Ev_X(\phi) = \left( \iota^*(\phi) \cup [(z_1-z_2)^k] \right) \cap [Y_\Q]
\]
and
\[
\Ev_{X'}(\phi) = \left( \iota^*(\phi) \cup [1_{X'}(z_1-z_2)^k] \right) \cap [Y_\Q].
\]

\subsection{Calculation with singular cohomology}

We will calculate this discrepancy by singular cohomology. For a topological space $S$, let $C_\bullet(S)$ be the singular chain complex, i.e., $C_i(S)$ is the free abelian group generated by singular $i$-simplices in $Y$, with the usual boundary maps.

Suppose $Y$ is an Eilenberg--MacLane space for $\Gamma$ with universal cover $H$. The natural action of $\Gamma$ on $H$ extends to an action on $C_\bullet(H)$, so that $C_\bullet(H)$ is equipped with the structure of a $\Gamma$-module. For any local system $M$ on $Y$, the cohomolgy $H^\bullet(Y, M)$ can be computed by the cochain complex
\[
\Hom_\Gamma(C_\bullet(H), M),
\]
and $H_c^\bullet(Y, M)$ can be computed by the compactly supported cochains, i.e., the cochains whose supports in $H$ are compact modulo $\Gamma$.

In our setting, $Y_K$ (resp. $Y_\Q$) is an Eilenberg--MacLane space for $\Gamma_K := \Gamma_0^K(N)$ (resp. $\Gamma_\Q := \Gamma_0^\Q(N)$). Thus every cohomology class in $H_c^2(Y_K, M)$ (resp. $H_c^2(Y_\Q, M)$) can be represented by a singular $2$-cochain in
\[
\Hom_{\Gamma_K}(C_2(\mc{H}_3), M) \quad \text{(resp. $\Hom_{\Gamma_\Q}(C_2(\mc{H}_2), M)$)}
\]
which is furthermore compactly supported.

\begin{lem}
Suppose $\phi \in H_c^2(Y_\Q, \mc{D}_k(\O))$ is represented by a compactly-supported singular $2$-cochain $\widetilde{\phi} \in \Hom_{\Gamma_K}(C_2(\mc{H}_3), \mc{D}_k(\O))$. Then
\[
\Ev_X(\phi) = \widetilde{\phi}(\iota([Y_\Q])) [(z_1-z_2)^k] \in \O
\]
and
\[
\Ev_{X'}(\phi) = \widetilde{\phi}(\iota([Y_\Q])) [1_{X'}(z_1-z_2)^k] \in \O.
\]
\end{lem}

\begin{proof}
Let $\sigma \in C_2(\mc{H}_2)$ denote any singular $2$-chain in $\mc{H}_2$. For emphasis, we always write $\iota(\sigma) \in C_2(\mc{H}_3)$ whenever $\sigma$ is thought of as lying inside $\mc{H}_3$.

At the level of singular cochains, the three operations above can be computed as follows. 
\begin{enumerate}
\item $\iota^* \phi \in H_c^2(Y_\Q, \mc{D}_k(\O))$ is represented by $\widetilde{\phi} \circ \iota \in \Hom_{\Gamma_\Q}(C_2(\mc{H}_2), \mc{D}_k(\O))$, i.e., the cochain
\[
\sigma \mapsto \widetilde{\phi}(\iota(\sigma)).
\]
\item Cup product is induced by evaluating $\mc{D}_k(\O)$ on the function $(z_1-z_2)^k \in \mc{C}_k(\O)$, so $(\iota^* \phi) \cup [(z_1-z_2)^k]$ is represented by the $2$-cochain
\[
\sigma \mapsto \widetilde{\phi}(\iota(\sigma)) [(z_1-z_2)^k].
\]
\item Finally, cap product with $[Y_\Q]$ corresponds to substituting $[Y_\Q]$ for $\sigma$ (this makes sense since the cochain $\widetilde{\phi}$ is compactly supported):
\[
\widetilde{\phi}(\iota([Y_\Q])) [(z_1-z_2)^k] \in \O.
\qedhere
\]
\end{enumerate}
\end{proof}

Now we recall the action of the Hecke operator $U_p$ on cohomology; here we do mean the \emph{rational} prime $p$, which is assumed to split in $K$. Suppose $M$ has a left action of $\Sigma_0(p)^2$, which restricts to an action of $\Gamma_0^K(N) \hookrightarrow \Sigma_0(p)^2$ via the embedding $\O_K \hookrightarrow \O_{k,p} = \Z_p \times \Z_p$. Then $H^q(Y_K, M)$ and $H_c^q(Y_K, M)$ are equipped with a Hecke action.

The $U_p$-operator is defined by the double coset $\Gamma_0^K(N) \delta \Gamma_0^K(N)$, where
\[
\delta := \left( \begin{pmatrix} p & 0 \\ 0 & 1 \end{pmatrix}, \begin{pmatrix} p & 0 \\ 0 & 1 \end{pmatrix} \right) \in \Sigma_0(p)^2.
\]
Choose a set of double coset representatives
\[
\gamma_{ij} := \left( \begin{pmatrix} p & i \\ 0 & 1 \end{pmatrix}, \begin{pmatrix} p & j \\ 0 & 1 \end{pmatrix} \right) \in \Sigma_0(p)^2
\]
for $i, j \in \{0, 1, \cdots, p-1\}$, so that
\[
\Gamma_0^K(N) \delta \Gamma_0^K(N) = \coprod_{i,j} \gamma_{ij} \Gamma_0^K(N) = \coprod_{i,j} \Gamma_0^K(N) \gamma_{ij}.
\]
Then we have the following description of the $U_p$-action, which is standard.

\begin{lem}
If $\phi \in H_*^q(Y_K, M)$ is represented by a cochain $\widetilde{\phi}$, then $U_p \phi$ is represented by
\[
\sigma \mapsto \sum_{i,j} \gamma_{ij} \cdot \widetilde{\phi}(\gamma_{ij}^{-1} \cdot \sigma).
\]
\end{lem}

Now we are ready to show the main result:

\begin{thm}
For $\phi \in H_c^2(Y_K, \mc{D}_k(\O))$,
\[
\Ev_X(U_p \phi) - \Ev_{X'}(U_p \phi) = p^{k+1} \Ev_X(\phi).
\]
\end{thm}

\begin{proof}
Suppose $\phi$ is represented by $\widetilde{\phi} \in \Hom_{\Gamma_K}(C_2(\mc{H}_3), M)$. We keep track of the $U_p$-action under the three successive operations.
\begin{enumerate}
\item $\iota^*(U_p \phi)$ is represented by
\[
\sigma \mapsto \sum_{i,j} \gamma_{ij} \cdot \widetilde{\phi}(\gamma_{ij}^{-1} \cdot \iota(\sigma)).
\]

\item $\iota^*(U_p \phi) \cup [(z_1-z_2)^k)]$ is represented by
\[
\sigma \mapsto \sum_{i,j} \gamma_{ij} \cdot \widetilde{\phi}(\gamma_{ij}^{-1} \cdot \iota(\sigma)) [(z_1-z_2)^k].
\]
By definition of the $\Sigma_0(p)^2$-actions on $\mc{D}_k(\O)$ and $\mc{C}_k(\O)$, this is equal to
\begin{align*}
& \sum_{i,j} \widetilde{\phi}(\gamma_{ij}^{-1} \cdot \iota(\sigma)) [\left. (z_1-z_2)^k \right|_k \gamma_{ij}] \\
=& \sum_{i,j} \widetilde{\phi}(\gamma_{ij}^{-1} \cdot \iota(\sigma)) [((pz_1+i)-(pz_2+j))^k].
\end{align*}

\item Finally, substituting $\sigma = [Y_\Q]$ gives
\begin{equation}
\label{eval-X}
\Ev_X(U_p \phi) = \sum_{i,j} \widetilde{\phi}(\gamma_{ij}^{-1} \cdot \iota([Y_\Q])) [((pz_1+i)-(pz_2+j))^k].
\end{equation}
\end{enumerate}

The evaluation $\Ev_{X'}(U_p \phi)$ is obtained in the same way, except that the function $(z_1-z_2)^k$ is replaced by $1_{X'}(z_1-z_2)^k$:
\begin{align*}
\Ev_{X'}(U_p \phi) &= \sum_{i,j} \widetilde{\phi}(\gamma_{ij}^{-1} \cdot \iota([Y_\Q])) [\left. 1_{X'}(z_1, z_2)(z_1-z_2)^k \right|_k \gamma_{ij}] \\
&= \sum_{i,j} \widetilde{\phi}(\gamma_{ij}^{-1} \cdot \iota([Y_\Q])) [1_{X'}(pz_1+i, pz_2+j)((pz_1+i)-(pz_2+j))^k].
\end{align*}
Since $1_{X'}$ is the characteristic function of $X' = \{(z_1, z_2): z_1-z_2 \in \Z_p^\times\}$, only the terms with $i \not\equiv j \pmod{p}$ remain, yielding
\begin{equation}
\label{eval-X'}
\Ev_{X'}(U_p \phi) = \sum_{i \not\equiv j \pmod{p}} \widetilde{\phi}(\gamma_{ij}^{-1} \cdot \iota([Y_\Q])) [((pz_1+i)-(pz_2+j))^k].
\end{equation}

Comparing (\ref{eval-X}) and (\ref{eval-X'}), we get
\begin{align*}
\Ev_X(U_p \phi) - \Ev_{X'}(U_p \phi) &= \sum_{i=0}^{p-1} \widetilde{\phi}(\gamma_{ii}^{-1} \cdot \iota([Y_\Q])) [((pz_1+i)-(pz_2+i))^k] \\
&= \sum_{i=0}^{p-1} \widetilde{\phi}(\gamma_{ii}^{-1} \cdot \iota([Y_\Q])) [p^k (z_1-z_2)^k] \\
&= p^{k+1} \widetilde{\phi}(\iota([Y_\Q])) [(z_1-z_2)^k]
\end{align*}
as $\gamma_{ii}^{-1}$ comes from an element of $\Gamma_0^\Q(N) \subset \SL_2(\Q)$, whose action fixes $\mc{H}_2 \subset \mc{H}_3$ and hence the fundamental class $[Y_\Q] \in H_2^\mathrm{BM}(Y_\Q, \Z)$. This concludes the proof.
\end{proof}

As a consequence, this establishes the following identity for a $U_p$-eigenclass:

\begin{cor}
\label{p-Euler-abs}
If $U_p \phi = \alpha \phi$, then
\[
\Ev_{X'}(\phi) = (1 - \alpha^{-1} p^{k+1}) \Ev_X(\phi).
\]
\end{cor}

\subsection{Summary}

\label{section-summary}

Composing all the evaluation maps defined in Section \ref{section-evaluation}, we obtain a big evaluation map $\mathbb{L}_\Lambda: H_c^2(Y_K, \mc{D}(\Lambda)) \to \Lambda$ fitting in the diagram
\[
\xymatrix{
H_c^2(Y_K, \mc{D}(\Lambda)) \ar[r]^-{\mathbb{L}_\Lambda} \ar[d]_{\rho_k} & \Lambda \ar[d]^{\sp_k} \\
H_c^2(Y_K, V_{k,k}(\O)) \ar[r]^-{\mathbb{L}_k} & \O
}
\]
which does \emph{not} commute; the failure to commute is made precise by the following theorem.

\begin{thm}
\label{p-Euler}
If $\Phi \in H_c^2(Y_K, \mc{D}(\Lambda))$ is a $U_p$-eigenclass with $U_p \Phi = \alpha \Phi$ (where $\alpha \in \Lambda$), then
\[
\sp_k(\mathbb{L}_\Lambda \Phi) = (1 - \sp_k(\alpha)^{-1} p^{k+1}) \mathbb{L}_k(\rho_k \Phi).
\]
\end{thm}

\begin{proof}
The specialization $\sp_k: H_c^2(Y_K, \mc{D}(\Lambda)) \to H_c^2(Y_K, \mc{D}_k(\O))$ is Hecke-equivariant, so $\sp_k(\Phi) \in H_c^2(Y_K, \mc{D}_k(\O))$ is an eigenclass with $U_p$-eigenvalue $\sp_k(\alpha)$. Now apply Corollary \ref{p-Euler-abs} to $\sp_k(\Phi)$.
\end{proof}

\section{The \texorpdfstring{$p$}{p}-adic \texorpdfstring{$L$}{L}-function on a Hida family}

\label{5}

In this section, we construct a $p$-adic $L$-function interpolating the special values $L(1, \ad(f) \otimes \alpha)$ as $f$ varies in a Hida family. Input from Hida theory will only be sketched, and the reader is strongly advised to refer to the original works of Hida.

\subsection{Base-change Hida families}

Consider the Iwasawa algebra $\Lambda_\Q := \O[[1+p\Z_p]]$. The points in $\Spec \Lambda_\Q(\overline{\Q_p})$ correspond to $\O$-algebra homomorphisms $\Lambda_\Q \to \overline{\Q_p}$, or equivalently continuous characters $1+p\Z_p \to \overline{\Q_p}^\times$. For $k \in \Z_{\geq 0}$ and a Dirichlet character $\epsilon$ of conductor $p^e$, let $P_{k,\epsilon} \in \Spec \Lambda_\Q(\overline{\Q_p})$ be the point induced by the character $\gamma \mapsto \epsilon(\gamma) \gamma^k$ for any topological generator $\gamma \in 1+p\Z_p$.

Let $\mb{h}_\Q$ be the universal ordinary (cuspidal) Hecke algebra for $\GL_2(\Q)$ defined in \cite{H86a}, \cite{H86b} (see the next paragraph for our choice of normalization), and $\mb{I}$ be a normal domain which is finite flat over $\Lambda_\Q$. Throughout we shall fix a Hida family $\lambda: \mb{h}_\Q \to \mb{I}$ with tame level $N$ and tame character $\chi: (\Z/Np\Z)^\times \to \O^\times$; equivalently, we can think of this as an ordinary $\mb{I}$-adic form $\mb{f} \in \mb{S}^\ord(N, \chi; \mb{I})$ \cite{Hida-book}.

Our convention is as follows. For any \emph{arithmetic} point $P \in \Spec \mb{I}(\overline{\Q_p})$, i.e., one that lies above $P_{k,\epsilon} \in \Spec \Lambda_\Q(\overline{\Q_p})$ for some $k \in \Z_{\geq 0}$ and $\epsilon$, the specialization $\lambda_P: \mb{h}_\Q \to \overline{\Q_p}^\times$ at $P$ corresponds to the system of Hecke eigenvalues given by an \emph{ordinary} eigenform $\mb{f}_P \in S_{k+2}(Np^e, \chi \epsilon \omega^{-k}; \overline{\Q_p})$. This normalization is unconventional, but is equivalent to the usual one and has the advantage of vastly simplifying notation for base-change considerations.

\begin{rem}
The Iwasawa algebra $\Lambda = \O[[\Z_p^\times]]$ is isomorphic to the product of $p-1$ copies of $\Lambda_\Q = \O[[1+p\Z_p]]$. Whenever a Hida family with tame character $\chi$ is given, we will implicitly fix the projection $\Lambda \twoheadrightarrow \Lambda_\Q$ that corresponds to the $p$-part of $\chi$ (a power of the Teichm\"uller character $\omega$).
\end{rem}

Let $\Lambda_K := \O[[\O_{K,p}^\times/(\text{torsion})]]$ be the Iwasawa algebra of the torsion-free part of $\O_{K,p}^\times$. In \cite{H94}, Hida defined the universal ordinary Hecke algebra $\mb{h}_K$ for $\GL_2(K)$ and shows that it is finite and torsion over $\Lambda_K$. Under the canonical base-change homomorphism $\mb{h}_K \to \mb{h}_\Q$ which is surjective, every Hida family on $\GL_2(\Q)$ lifts to a Hida family on $\GL_2(K)$.

\begin{defn}
Given a Hida family $\lambda: \mb{h}_\Q \to \mb{I}$ for $\GL_2(\Q)$, its base-change Hida family for $\GL_2(K)$ is denoted by $\lambda^K: \mb{h}_K \to \mb{I}$.
\end{defn}

At an arithmetic point $P$ lying above $P_{k, \epsilon}$, $\lambda^K$ specializes to the base-change of $\lambda_P$. Under our normalization, $\mb{f}_P$ has weight $k+2$ and its base-change $\BC(\mb{f}_P)$ has weight $(k,k)$.

\begin{rem}
Geometrically, $\lambda^K$ corresponds to an irreducible component of $\mb{h}_K$ which is supported over the parallel weights on $\O_{K,p}^\times$. 
\end{rem}

\subsection{Control theorem}

In \cite{H93} and \cite{H94}, Hida proved control theorems for the ordinary parts of the $p^\infty$-level cohomology group and universal Hecke algebra for $\GL_2$ over an arbitrary number field. In the base-change situation, we have the following:

\begin{thm}
\label{mult-one}
Suppose $\lambda: \mb{h}_\Q \to \mb{I}$ is a Hida family. Then the $\lambda^K$-eigenspace for $\mb{h}_K$ in
\[
H_{c, \ord}^2(Y_K, \mc{D}(\Lambda)) \otimes_\Lambda \Frac(\mb{I})
\]
is one-dimensional.
\end{thm}

\begin{proof}[Proof (sketch)]
By \cite{H94}, every irreducible component of $\mb{h}_K$ occurs in $H_{c, \ord}^2(Y_K, \mc{D}(\Lambda_K))$. Components which are base-changed via $\mb{h}_K \to \mb{h}_\Q$ factor through $\Lambda_K \twoheadrightarrow \Lambda_\Q$ corresponding to the closed subscheme of $\Spec \Lambda_K$ of parallel weights.

Strong multiplicity one for $\GL_2(K)$ implies that the eigenspace is one-dimensional.
\end{proof}

By Theorem \ref{mult-one}, we may choose a $\Frac(\mb{I})$-basis $\mc{F}$ of the $\lambda^K$-eigenspace, which realizes the base-change Hida family as a Hecke eigenclass in the cohomology $H_{c, \ord}^2(Y_K, \mc{D}(\Lambda)) \otimes_\Lambda \Frac(\mb{I})$. Roughly speaking, we can apply the weight $k$ specialization map
\[
\rho_k: H_c^2(Y_K, \mc{D}(\Lambda)) \to H_c^2(Y_K, V_{k,k}(\O))
\]
to get
\[
\rho_k(\mc{F}) \in H_c^2(Y_K, V_{k,k}(\O)) \otimes_\O \overline{\Q_p}.
\]
Comparing this with the basis elements $\widehat{\delta}_2$ defined in Proposition \ref{period} yields $p$-adic error terms which measure how far each $\rho_k(\mc{F})$ is from being integral. There are two minor issues with this, however:

\begin{enumerate}
\item We have to avoid the primes $P$ which divide the denominator of $\mc{F}$.
\item We have only set up the notation for Bianchi forms with \emph{trivial} Nebentype.
\end{enumerate}
To address (1), we simply need to avoid finitely many primes. Although (2) poses a serious condition on the tame character of the given Hida family, it will be compatible with our setting. Thus we content ourselves with an \textit{ad hoc} definition.

\begin{defn}
\label{error}
Suppose $\lambda: \mb{h}_\Q \to \mb{I}$ is a Hida family, and $\mc{F}$ is a $\Frac(\mb{I})$-basis of the $\lambda^K$-eigenspace of $H_{c, \ord}^2(Y_K, \mc{D}(\Lambda)) \otimes_\Lambda \Frac(\mb{I})$ by Theorem \ref{mult-one}. Then for all arithmetic points $P \in \Spec \mb{I}(\overline{\Q_p})$ of weight $k = k_P \in \Z_{\geq 2}$ such that:
\begin{enumerate}
\item $\mc{F}$ does not have a pole at $P$;
\item the base-change of $\mb{f}_P$ has \emph{trivial} Nebentype;
\end{enumerate}
the $p$-adic error term $c_P(\mc{F}) \in \overline{\Q_p}$ is defined by the relation
\[
\rho_k(\mc{F}) = c_P(\mc{F})\cdot \widehat{\delta}_2(\BC(\mb{f}_P))
\]
in $H_c^2(Y_K, V_{k,k}(\O)) \otimes_\O \overline{\Q_p}$. Note that $c_P(\mc{F}) \in \overline{\Q_p}^\times$ for all but finitely many such $P$.
\end{defn}

In the absence of additional hypotheses, it seems difficult to control these error terms $c_P$ as $P$ varies over the arithmetic points, but recent breakthroughs in modularity lifting might allow us to obtain better control. This is similar to the difference between \cite{GS93} and \cite{K94}; the former obtains no control over the $p$-adic error terms, and the latter imposes Gorenstein-type conditions to ensure the $p$-adic error terms are units.

\subsection{Construction of the \texorpdfstring{$p$}{p}-adic \texorpdfstring{$L$}{L}-function}

Finally, we are ready to construct the $p$-adic $L$-function by evaluating the big evaluation map $\mathbb{L}_\Lambda$ on a Hecke eigenclass realizing the base-change Hida family.
\[
\xymatrix{
H_c^2(Y_K, \mc{D}(\Lambda)) \ar[r]^-{\mathbb{L}_\Lambda} \ar[d]_{\rho_k} & \Lambda \ar[d]^{\sp_k} \\
H_c^2(Y_K, V_{k,k}(\O)) \ar[r]^-{\mathbb{L}_k} & \O
}
\]
For any classical eigenform $f$ with level divisible by $p$, we denote by $a_p(f)$ its $U_p$-eigenvalue.

\begin{thm}
\label{main-thm}
Let $\lambda: \mb{h}_\Q \to \mb{I}$ be a Hida family with tame character $\alpha \omega^r$, and $\mc{F}$ be a $\Frac(\mb{I})$-basis of the $\lambda^K$-eigenspace of $H_{c, \ord}^2(Y_K, \mc{D}(\Lambda)) \otimes_\Lambda \Frac(\mb{I})$. Let $A_r \subset \Spec \mb{I}(\overline{\Q_p})$ be the set of arithmetic points lying above $P_k$ with $k \equiv r \pmod{p-1}$. Then there exists $\mc{L} \in \Frac(\mb{I})$ such that
\[
\mc{L}(P) = c_P(\mc{F}) (1 - a_p(\mb{f}_P)^{-2} p^{k+1}) L^\alg(1, \ad(\mb{f}_P) \otimes \alpha)
\]
for almost all $P \in A_r$.
\end{thm}

\begin{proof}
The definition of $A_r$ ensures that at an arithmetic point $P \in A_r$ of weight $k$, the Hida family specializes to $\mb{f}_P \in S_{k+2}(Np, \alpha)$, whose base change $\BC(\mb{f}_P)$ has weight $(k,k)$ and trivial Nebentype.

After extending scalars, we have
\[
\mathbb{L}_\Lambda: H_c^2(Y_K, \mc{D}(\Lambda)) \otimes_{\Lambda} \Frac(\mb{I}) \to \Frac(\mb{I})
\]
and define
\[
\mc{L} := \mathbb{L}_\Lambda(\mc{F}) \in \Frac(\mb{I}).
\]
Since $p$ is split in $K$, the $U_p$-eigenvalue of $\BC(\mb{f}_P)$ is equal to $a_p(\mb{f}_P)^2$. Then the interpolation formula follows from Theorem \ref{p-Euler} and the definition of $c_P(\mc{F})$.
\end{proof}

\begin{rem}
For eigenforms $f$ with Nebentype not equal to $\alpha$, Hida's integral formula \cite{H99} for $L(1, \ad(f) \otimes \alpha)$ involves twisting the base-change Bianchi form $f_K$ by a suitable Hecke character $\varphi: \A_K^\times/K^\times \to \C^\times$. Accordingly, twisting the evaluation map $\mathbb{L}_\Lambda: H_c^2(Y_K, \mc{D}(\Lambda)) \to \Lambda$ by an appropriate character $\varphi$ will give a $p$-adic $L$-function on a different Hida family, but the set of weights at which we can determine the specialization will simply be a translate of $A_r$ as in Theorem \ref{main-thm}; in particular, all the eigenforms $f$ will have the same Nebentype. This seems to present a genuine difficulty with extending the interpolation formula to a larger domain of weights.
\end{rem}

\bibliographystyle{alpha}
\bibliography{Adjoint}

\providecommand{\noopsort}[1]{}
\begin{thebibliography}{BSW18}

\bibitem[AS86]{AS}
Avner Ash and Glenn Stevens.
\newblock Modular forms in characteristic {$l$} and special values of their
  {$L$}-functions.
\newblock {\em Duke Math. J.}, 53(3):849--868, 1986.

\bibitem[Bel]{Bellaiche}
Jo\"{e}l Bella\"{i}che.
\newblock The eigenbook: Eigenvarieties, families of {G}alois representations,
  {$p$}-adic {$L$}-functions.
\newblock To appear in Pathways in Mathematics. Birkh\"{a}user--Springer.

\bibitem[BSW18]{BW-Bianchi}
Daniel Barrera~Salazar and Chris{\noopsort{b}} Williams.
\newblock Families of {B}ianchi modular symbols: critical base-change
  {$p$}-adic {$L$}-functions and {$p$}-adic {A}rtin formalism.
\newblock With an appendix by Carl Wang-Erickson. Preprint, 2018.

\bibitem[Gha99]{G99}
Eknath Ghate.
\newblock Critical values of the twisted tensor {$L$}-function in the imaginary
  quadratic case.
\newblock {\em Duke Math. J.}, 96(3):595--638, 1999.

\bibitem[GJ78]{GJ}
Stephen Gelbart and Herv\'{e} Jacquet.
\newblock A relation between automorphic representations of {${\rm GL}(2)$} and
  {${\rm GL}(3)$}.
\newblock {\em Ann. Sci. \'{E}cole Norm. Sup. (4)}, 11(4):471--542, 1978.

\bibitem[GS93]{GS93}
Ralph Greenberg and Glenn Stevens.
\newblock {$p$}-adic {$L$}-functions and {$p$}-adic periods of modular forms.
\newblock {\em Invent. Math.}, 111(2):407--447, 1993.

\bibitem[Hid81a]{H81a}
Haruzo Hida.
\newblock Congruence of cusp forms and special values of their zeta functions.
\newblock {\em Invent. Math.}, 63(2):225--261, 1981.

\bibitem[Hid81b]{H81b}
Haruzo Hida.
\newblock On congruence divisors of cusp forms as factors of the special values
  of their zeta functions.
\newblock {\em Invent. Math.}, 64(2):221--262, 1981.

\bibitem[Hid86a]{H86a}
Haruzo{\noopsort{a}} Hida.
\newblock Iwasawa modules attached to congruences of cusp forms.
\newblock {\em Ann. Sci. \'{E}cole Norm. Sup. (4)}, 19(2):231--273, 1986.

\bibitem[Hid86b]{H86b}
Haruzo{\noopsort{b}} Hida.
\newblock Galois representations into {${\rm GL}_2({\bf Z}_p[[X]])$} attached
  to ordinary cusp forms.
\newblock {\em Invent. Math.}, 85(3):545--613, 1986.

\bibitem[Hid88]{H88b}
Haruzo{\noopsort{b}} Hida.
\newblock Modules of congruence of {H}ecke algebras and {$L$}-functions
  associated with cusp forms.
\newblock {\em Amer. J. Math.}, 110(2):323--382, 1988.

\bibitem[Hid93a]{Hida-book}
Haruzo Hida.
\newblock {\em Elementary theory of {$L$}-functions and {E}isenstein series},
  volume~26 of {\em London Mathematical Society Student Texts}.
\newblock Cambridge University Press, Cambridge, 1993.

\bibitem[Hid93b]{H93}
Haruzo Hida.
\newblock {$p$}-ordinary cohomology groups for {${\rm SL}(2)$} over number
  fields.
\newblock {\em Duke Math. J.}, 69(2):259--314, 1993.

\bibitem[Hid94a]{H94-L}
Haruzo Hida.
\newblock On the critical values of {$L$}-functions of {${\rm GL}(2)$} and
  {${\rm GL}(2)\times{\rm GL}(2)$}.
\newblock {\em Duke Math. J.}, 74(2):431--529, 1994.

\bibitem[Hid94b]{H94}
Haruzo Hida.
\newblock {$p$}-adic ordinary {H}ecke algebras for {${\rm GL}(2)$}.
\newblock {\em Ann. Inst. Fourier (Grenoble)}, 44(5):1289--1322, 1994.

\bibitem[Hid99]{H99}
Haruzo Hida.
\newblock Non-critical values of adjoint {$L$}-functions for {${\rm SL}(2)$}.
\newblock In {\em Automorphic forms, automorphic representations, and
  arithmetic ({F}ort {W}orth, {TX}, 1996)}, volume~66 of {\em Proc. Sympos.
  Pure Math.}, pages 123--175. Amer. Math. Soc., Providence, RI, 1999.

\bibitem[Jac72]{Jacquet-2}
Herv\'{e} Jacquet.
\newblock {\em Automorphic forms on {${\rm GL}(2)$}. {P}art {II}}.
\newblock Lecture Notes in Mathematics, Vol. 278. Springer-Verlag, Berlin-New
  York, 1972.

\bibitem[Kim06]{Kim}
Walter Kim.
\newblock {\em Ramification points on the eigencurve and the two variable
  symmetric square p-adic {L}-function}.
\newblock ProQuest LLC, Ann Arbor, MI, 2006.
\newblock Thesis (Ph.D.)--University of California, Berkeley.

\bibitem[Kit94]{K94}
Koji Kitagawa.
\newblock On standard {$p$}-adic {$L$}-functions of families of elliptic cusp
  forms.
\newblock In {\em {$p$}-adic monodromy and the {B}irch and {S}winnerton-{D}yer
  conjecture ({B}oston, {MA}, 1991)}, volume 165 of {\em Contemp. Math.}, pages
  81--110. Amer. Math. Soc., Providence, RI, 1994.

\bibitem[Lan80]{Langlands-BC}
Robert~P. Langlands.
\newblock {\em Base change for {${\rm GL}(2)$}}, volume~96 of {\em Annals of
  Mathematics Studies}.
\newblock Princeton University Press, Princeton, N.J.; University of Tokyo
  Press, Tokyo, 1980.

\bibitem[Rib83]{R83}
Kenneth~A. Ribet.
\newblock Mod {$p$} {H}ecke operators and congruences between modular forms.
\newblock {\em Invent. Math.}, 71(1):193--205, 1983.

\bibitem[Shi75]{S75}
Goro Shimura.
\newblock On the holomorphy of certain {D}irichlet series.
\newblock {\em Proc. London Math. Soc. (3)}, 31(1):79--98, 1975.

\bibitem[TU18]{TU18}
Jacques Tilouine and Eric Urban.
\newblock Integral period relations and congruences.
\newblock Preprint, 2018.

\bibitem[Urb95]{U95}
\'{E}ric Urban.
\newblock Formes automorphes cuspidales pour {${\rm GL}_2$} sur un corps
  quadratique imaginaire. {V}aleurs sp\'{e}ciales de fonctions {$L$} et
  congruences.
\newblock {\em Compositio Math.}, 99(3):283--324, 1995.

\bibitem[Wil17]{W17}
Chris Williams.
\newblock {$P$}-adic {$L$}-functions of {B}ianchi modular forms.
\newblock {\em Proc. Lond. Math. Soc. (3)}, 114(4):614--656, 2017.

\end{thebibliography}

\end{document}